\documentclass[preprint]{elsarticle}
\usepackage{amsmath,amsfonts,amsthm,url,color,amssymb}
\usepackage{verbatim}
\usepackage{enumerate}
\usepackage{url}
\usepackage{tikz}
\newtheorem{theorem}{Theorem}[section]
\newtheorem{proposition}[theorem]{Proposition}
\newtheorem{lemma}[theorem]{Lemma}
\newtheorem{corollary}[theorem]{Corollary}
\newtheorem{definition}[theorem]{Definition}

\newtheorem{example}[theorem]{Example}
\newcommand{\mdeg}{\mathrm{mdeg}}
\newcommand{\A}{\mathbf{A}}
\newcommand{\B}{\mathbf{B}}
\newcommand{\x}{\mathbf{X}}
\newcommand{\y}{\mathbf{Y}}
\newcommand{\F}{\mathbb F}

\newcommand{\N}{\mathbb N}

\newcommand{\lc}{\mathrm{LC}}
\newcommand{\lt}{\mathrm{LT}}
\newcommand{\lm}{\mathrm{LM}}

\newcommand{\GL}{\mathrm{GL}}

\newcommand{\ord}{\mathrm{ord}}

\newcommand{\doublespace}

\begin{document}

\begin{frontmatter}

\title{A group action on multivariate polynomials over finite fields}
\author{Lucas Reis \fnref{fn1}}
\ead{lucasreismat@gmail.com}
\fntext[fn1]{Permanent address: Departamento de Matem\'{a}tica, Universidade Federal de Minas Gerais, UFMG, Belo Horizonte, MG, 30123-970, Brazil}
\address{School of Mathematics and Statistics, Carleton University, 1125 Colonel By Drive, Ottawa ON (Canada), K1S 5B6}

\begin{abstract}
Let $\F_q$ be the finite field with $q$ elements, where $q$ is a power of a prime $p$. Recently, a particular action of the group $\GL_2(\F_q)$ on irreducible polynomials in $\F_q[x]$ has been introduced and many questions concerning the invariant polynomials have been discussed. In this paper, we give a natural extension of this action on the polynomial ring $\F_q[x_1, \ldots, x_n]$ and study the algebraic properties of the invariant elements.
\end{abstract}

\end{frontmatter}

\section{Introduction}
Let $\F_q$ be a finite field with $q$ elements, where $q$ is a power of a prime $p$. Given $A\in \GL_2(\F_q)$, the matrix $A$ induces a natural map on $\F_q[x]$. Namely, if we write
$A=
\left(\begin{matrix}
a&b\\
c&d
\end{matrix}\right),
$
given $f(x)$ of degree $n$ we define $$A\diamond f=(cx+d)^nf\left( \frac{ax+b}{cx+d}\right).$$

It turns out that, when restricted to the set $I_n$ of irreducible polynomials of degree $n$ (for $n\ge 2$), this map is a permutation of $I_n$ and, more than that, $\GL_2(\F_q)$ acts on $I_n$ via the compositions $A\diamond f$. This was first noticed by Garefalakis in \cite{Gar11}. Recently, this action (and others related) has attracted attention from several authors (see \cite{ST12}, \cite{LR17} and \cite{Kap17}), and some fundamental questions have been discussed such as the characterization and number of invariant  irreducible polynomials of a given degree. The map induced by $A$ preserves the degree of elements in $I_n$ (for $n\ge 2$), but not in the whole ring $\F_q[x]$: for instance, $A=
\left(\begin{matrix}
1&1\\
1&0
\end{matrix}\right)
$ is such that $A\diamond (x^n-1)=(x+1)^n-x^n$ has degree at most $n-1$. However, if the ``denominator'' $cx+d$ is trivial, i.e., $c=0$ and $d=1$, the map induced by $A$ preserves the degree of any polynomial and, more than that, is an $\F_q$-automorphism of $\F_q[x]$. This motivates us to introduce the following: let $\mathcal A_n:=\F_q[x_1,\ldots, x_n]$ be the ring of polynomials in $n$ variables over $\F_q$ and $G$ be the subgroup of $\GL_2(\F_q)$ comprising the elements of the form
$$A=
\left(\begin{matrix}a&b\\0&1\end{matrix}\right).
$$
The set $G^n:=\underbrace{G\times \cdots \times G}_{n\,\mathrm{times}}$, equipped with the coordinate-wise product induced by $G$, is a group. The group $G^n$ induces $\F_q$-endomorphisms of $\mathcal A_n$; given $\A\in G^n$, $\A=(A_1, \ldots, A_n)$, where $A_i=\left(\begin{matrix}a_i&b_i\\0&1\end{matrix}\right)$, and $f\in \mathcal A_n$, we define
$$\A\circ f:=f(A_1 \circ x_1, \ldots, A_n\circ x_n)= f(a_1x_1+b_1, \ldots, a_nx_n+b_n)\in \mathcal A_n.$$
In other words, $\A$ induces the $\F_q$-endomorphism of $\mathcal A_n$ given by the substitutions $x_i\mapsto a_ix_i+b_i$. We will show that this map induced by $\A$ is an $\F_q$-automorphism of $\mathcal A_n$ and, in fact, this is an action of $G^n$ on the ring $\mathcal A_n$, such that $\A \circ f$ and $f$ have the same {\it multidegree} (a natural extension of degree in several variables). It is then natural to explore the algebraic properties of the fixed elements. We define $R_{\A}$ as the subring of $\mathcal A_n$ comprising the polynomials invariant by $\A$, i.e.,

$$R_{\A}:=\{f\in \mathcal A_n\, |\, \A\circ f=f \}.$$

The ring $R_{\A}$ is frequently called the {\it fixed-point} subring of $\mathcal A_n$ by $\A$. The study of the fixed-point subring plays an important role in the {\it Invariant Theory of Polynomials}. Note that $R_{\A}$ is an $\F_q$-algebra and then some interesting theoretical questions arise:

\begin{enumerate}[$\bullet$]
\item Is $R_{\A}$ a finitely generated $\F_q$-algebra? If yes, can we find a minimal set $S$ of generators? What about the size of $S$?

\item Is $R_{\A}$ a free $\F_q$-algebra? That is, can $R_{\A}$ be viewed as a polynomial ring in some number of variables?
\end{enumerate}

Any polynomial is invariant by $\A$ if and only if is invariant by any element of the group $\langle \A\rangle$ generated by $\A$ in $G^n$. In particular, we can explore the fixed-point subring for any subgroup $H$ of $G^n$.

For $n=1$, the equality $\A\circ f=f$ becomes $f(x)=f(ax+b)$ for some $a\in \F_q^*$ and $b\in \F_q$. In other words, we are taking the substitution $x\mapsto ax+b$. It turns out that, with an affine change of variable, we are able to reduce to the cases of translations $x\mapsto x+b$ and the homotheties $x\mapsto ax$. In those cases, the fixed-point subring is well understood and we can easily answer the questions above (see Theorems 2.5 and 3.1 of \cite{LR17}).

In this paper we discuss those questions for any $n\ge 1$. We show that $R_{\A}$ is always a finitely generated $\F_q$-algebra, we find explicitly a minimal set of generators $S$ and show that the size of such $S$ is related to the number of some special minimal product-one sequences in the multiplicative group $\F_{q}^*$. Also, we give necessary and sufficient conditions on the element $\A$ for $R_{\A}$ to be a free $\F_q$-algebra. 

The paper is structured as follows. In Section 2 we recall some basic theory of multivariate polynomials over commutative rings and present some preliminary results. Section 3 is devoted to characterize the ring $R_{\A}$ and find a minimal set $S$ of generators. In Section 4 we find estimates for the size of $S$ and, in Section 5, we study the fixed-point subring by the action of $H$ of $G^n$, where $H$ is any Sylow subgroup of $G^n$. Finally, in Section 6, we consider the subgroup $\mathcal H$ of $G^2$, comprising the elements whose coordinates are all diagonal matrices; in this case, we obtain an alternative characterization of the fixed elements of $\F_q[x, y]$, looking at its homogeneous components. In particular, we obtain a correspondence between the homogeneous that are fixed by elements of $\mathcal H$ and the univariate polynomials through the action studied in \cite{Gar11}. In fact, $\mathcal H$ can be viewed as a subgroup of $\GL_2(\F_q)$ and we show that this correspondence extends in a more general action of $\GL_2(\F_q)$ on $\F_q[x, y]$. 

\section{Preliminaries }
Throughout this paper, $\F_q$ denotes the finite field with $q$ elements, where $q$ is a power of a prime $p$ and $\mathcal A_n:=\F_q[x_1, \ldots, x_n]$ denotes the ring of polynomials in $n$ variables over $\F_q$. Also, for elements $a\in \F_q^*$, $A\in \GL_2(\F_q)$ and $\A\in G^n$, we denote by $\ord(a)$, $\ord(A)$ and $\ord(\A)$ the multiplicative orders of $a$, $A$ and $\A$, respectively.

As mentioned before, the univariate polynomials that remains invariant by the substitution $x\mapsto ax+b$ are well described and, for completeness, we state the results:
\begin{theorem}\label{LR}
Suppose that $f(x)$ is a polynomial in $\F_q$. Then the following hold:
\begin{enumerate}[(i)]
\item $f(x+b)=f(x)$ if and only if $f(x)=g(x^p-b^{p-1}x)$ for some $g(x)\in \F_q[x]$.
\item $f(ax)=f(x)$ if and only if $f(x)=g(x^k)$ for some $g(x)\in \F_q[x]$, where $k=\ord(a)$.
\end{enumerate}
\end{theorem}
For the proof of this result, see Theorems 2.5 and 3.1 of \cite{LR17}. The case $a\ne 1$ and $b\ne 0$ can be reduced to the case $b=0$. In fact, we have $f(ax+b)=f(x)$ if and only if $f_0(ax)=f_0(x)$, where $f_0(x)=f\left(x-\frac{b}{a-1}\right)$. This kind of trick will be used frequently to simplify some calculations.

From Theorem \ref{LR}, we can deduce that the fixed-point subrings are $$\F_q[x^p-b^{p-1}x]\quad\text{and}\quad\F_q[y^k],$$ where $y=x+\frac{b}{a-1}$ or $y=x$. Clearly these rings are finitely generated $\F_q$-algebras and isomorphic to $\F_q[z]$, the ring of univariate polynomials over $\F_q$. Also, we have explicitly a minimal set of generators $S$ with size one. As we will see, in general, the sets $S$ are more complex. 

We start with some basic theory on multivariate polynomials over commutative rings. For more details, see Chapter 2 of \cite{Cox}. Throughout this paper we always consider the {\it graded lexicographical} order in $\mathcal A_n$, denoted by $<_{\text{gradlex}}$, such that
$$x_1>x_2>\cdots>x_n.$$

For a given monomial in $\mathcal A_n$, say $x_1^{\alpha_1}\ldots x_n^{\alpha_n}$, we write $\x^{\alpha}$, where $\alpha=(\alpha_1, \ldots, \alpha_n)\in \N^n$. For convention, $x_i^{0}=1$.  Sometimes, we simply write $\x$ or $\y$ for a generic monomial in $\mathcal A_n$.

It turns out that the graded lex order is induced by the following ordering of the vectors $\alpha\in \N^n$: given two elements $\alpha=(\alpha_1, \ldots, \alpha_n)$ and $\alpha'=(\alpha_1', \ldots, \alpha_n')$ we say that $\alpha>\alpha'$ if $\sum \alpha_i>\sum \alpha_i'$ or $\sum \alpha_i=\sum \alpha_i'$ and the leftmost nonzero coordinate of the  difference vector $\alpha-\alpha'$ is positive. In this case, we write $\x^{\alpha}>\x^{\alpha'}$.

\begin{example}
Let $\F_q[x, y, z]$ be the ring of polynomials in three variables and $\x^{\alpha}=x^2y^3$, $\x^{\alpha'}=xy^2z^3$ and $\x^{\alpha''}=x^2y^2z^2$. Consider the graded lexicographical order such that $x>y>z$. We have $$x^2y^3<_{\mathrm{gradlex}}xy^2z^3<_{\mathrm{gradlex}} x^2y^2z^2.$$
\end{example}

Any nonzero polynomial $f\in \mathcal A_n$ can be written uniquely as $\sum_{\alpha\in B}a_{\alpha}\x^{\alpha}$ for some non zero elements $a_{\alpha}\in \F_q$ and a finite set $B$.

\begin{definition}
Let $f$ be any nonzero element of $\mathcal A_n$. The {\it multidegree} of $f$ is the maximum of $\alpha$ (with respect to the graded lexicographical order) such that $\alpha\in B$. For simplicity, we write $\alpha=\mdeg f$.
\end{definition}

\subsection{A natural action of $G^n$ over $\mathcal A_n$}
As follows, the compositions $\A\circ f$ have some basic properties.
\begin{lemma}\label{action}
Given $\A\in G^n$, $\A=(A_1, \ldots, A_n)$, where $A_i=\left(\begin{matrix}a_i&b_i\\0&1\end{matrix}\right)$, and $f, g\in \mathcal A_n$. The following hold:

\begin{enumerate}[a)]
\item For any non zero element $f\in \mathcal A_n$, the polynomials $f$ and $\A\circ f$ have the same multidegree.

\item If\, $\mathbf I$ denotes the identity element of $G^n$, $\mathbf I\circ f=f$.

\item $\A\circ (f\cdot g)=(\A\circ f)\cdot (\A\circ g)$. In particular, $\A\circ f$ is irreducible if and only if $f$ is irreducible.

\item Given $\A'\in G^n$, $\A'=(A_1', \ldots, A_n')$ and $A_i'=\left(\begin{matrix}a_i'&b_i'\\0&1\end{matrix}\right)$, we have
$$(\A'\A)\circ f=\A'\circ (\A\circ f).$$
In particular, the endomorphism induced by $\A$ is an $\F_q$-automorphism of $\mathcal A_n$, with inverse induced by $\A^{-1}$.

\item The automorphism induced by $\A$ is of finite order and its order coincides with the order of $\A$ in $G^n$.
\end{enumerate}
\end{lemma}

\begin{proof}
\begin{enumerate}[a)]
\item Since $\A\circ f$ is linear and $\deg(f+g)=\max(\deg f, \deg g)$ for any elements $f, g\in \mathcal A_n$ of distinct multidegree, we just have to prove the statement for monomials. Notice that, writing $\alpha=(\alpha_1, \ldots, \alpha_n)$, $$\A\circ(\x^{\alpha})=(a_1x_1+b_1)^{\alpha_1}\cdots (a_nx_n+b_n)^{\alpha_n}=a_1\ldots a_n \cdot \x^{\alpha}+\sum_{\beta \in B}a_{\beta} \x^{\beta},$$ where $\beta<\alpha$ for any $\beta \in B$ and $a_{\beta}\in \F_q$. Since $a_1\ldots a_n\ne 0$, we have that the multidegree of $\A\circ(\x^{\alpha})$ is $\alpha$ and this finishes the proof.

\item This follows directly by definition, since $\mathbf I=(I, \ldots, I)$, where $I$ is the identity element of $\GL_2(\F_q)$.

\item This follows directly by calculations. 

\item Notice that $$\A'\circ (\A\circ f)=\A'\circ f(a_1x_1+b_1, \ldots, a_nx_n+b_n)=f(cx_1+d_1,\ldots, c_nx_n+d_n),$$
where $c_i=a_ia_i'$ and $d_i=a_i'b_i+b_i'$ for any $1\le i\le n$. By a directly calculation we see that $\A'\cdot \A=(B_1, \ldots, B_n)$, where $B_i=\left(\begin{matrix}c_i&d_i\\0&1\end{matrix}\right)$, i.e., $(\A'\A)\circ f=\A'\circ (\A\circ f)$. Note that $\A^{-1}=(A_1^{-1}, \ldots, A_n^{-1})$ and then 
$$\A^{-1}\circ (\A\circ f)=(\A^{-1}\cdot \A)\circ f= \mathbf {I}\circ f=f.$$ In particular, the endomorphism induced by $\A$ is an automorphism of $\mathcal A_n$, with inverse induced by $\A^{-1}$.

\item From the previous item, $$\underbrace{\A\circ \ldots\circ  \A}_{d \, \mathrm{times}}\circ f=\A^d\circ f.$$ To conclude the proof, notice that $\A^d$ induces the identity map if and only if $\A^d=\mathbf {I}$ in $G^n$ and the minimal positive integer $d$ with this property is $d=\ord(\A)$.
\end{enumerate}
\end{proof}

Lemma \ref{action} says that $G^n$ acts on $\mathcal A_n$ via the compositions $\A\circ f$. From now, $\A$ denotes an element of $G^n$ and the automorphism of $\mathcal A_n$ induced by it.  We also know that the automorphism $\A$ has order $\ord(\A)$. How large can be this order?

\begin{lemma}
The group $G^n$ has $[q(q-1)]^n$ elements, any of them of order dividing $p(q-1)$. Moreover, for $n>1$, there exist an element of order $p(q-1)$.
\end{lemma}

\begin{proof}
Clearly $G$ has $q(q-1)$ elements (since $b\in \F_q$ and $a\in \F_q^*$), hence $G^n$ has $[q(q-1)]^n$ elements. For a generic element $A\in G$ with $A=\left(\begin{matrix}a&b\\0&1\end{matrix}\right)$, we have $\ord(A)=\ord(a)$ if $a\ne 1$ or $b=0$ and $\ord(A)=p$, otherwise. In particular, the order of any element $A\in G$ divides $p$ or $(q-1)$. Since the order of $\A$ is just the least common multiple of the order of its coordinates (viewed as elements of $G$), the order of $\A$ in $G^n$ always divides $p(q-1)$. Also, since $\F_q^*$ is cyclic, there is an element $\theta\in \F_q$ of order $q-1$. A direct calculation shows that, if $n>1$, the element 
$$\A_0=\left[\left(\begin{matrix}\theta &0\\0&1\end{matrix}\right),\left(\begin{matrix}1&1\\0&1\end{matrix}\right), I, \ldots, I\right],$$
has order $p(q-1)$ in $G^n$.
\end{proof}

We have noticed that, in the univariate case, the study of invariant polynomials can be reduced to the study of translations $x\mapsto x+b$ and homotheties $x\mapsto ax$. The idea relies on the changes of variable $y=x+\frac{b}{a-1}$ and $y=\frac{x}{b}$. In terms of matrices, we are just taking conjugations; we will see that this can be extended more generally. For an element $A\in G$ we say that $A$ is of $h$-{\it type} or $t$-{\it type} if $A=\left(\begin{matrix}a&0\\0&1\end{matrix}\right)$ for some $a\ne 0, 1$ or $A=\left(\begin{matrix}1&1\\0&1\end{matrix}\right)$, respectively. We can easily see that any element of $G$ different from the identity is conjugated in $G$ to an element of $h$-type or $t$-type. The first case occurs when $A$ is diagonalizable ($a\ne 1$) and the second one occurs when $A$ has only ones in the diagonal:
$$A=\left(\begin{matrix}a&b\\0&1\end{matrix}\right)=\left(\begin{matrix}1&\frac{-b}{a-1}\\0&1\end{matrix}\right)\left(\begin{matrix}a&0\\0&1\end{matrix}\right)\left(\begin{matrix}1&\frac{b}{a-1}\\0&1\end{matrix}\right),$$
and
$$A=\left(\begin{matrix}1&b\\0&1\end{matrix}\right)=\left(\begin{matrix}b&0\\0&1\end{matrix}\right)\left(\begin{matrix}1&1\\0&1\end{matrix}\right)\left(\begin{matrix}\frac{1}{b}&0\\0&1\end{matrix}\right).$$

We obtain the following.

\begin{theorem}\label{HT}
Let $\A$ and $\B$ two elements in $G^n$ that are conjugated, $\A=\A_0\B\A_0^{-1}$, where $\A_0\in G^n$. The following hold:
\begin{enumerate}[a)]
\item The $\F_q$-automorphism induced by $\A_0^{-1}$, when restricted to $R_{\A}$, is an $\F_q$-isomorphism between $R_{\A}$ and $R_{\B}$. Moreover, if $R_{\A}$ is a finitely generated $\F_q$-algebra such that $R_{\A}=\F_q[f_1, \ldots, f_m]$, where $f_i\in \mathcal A_n$, then $R_{\B}$ is also finitely generated as an $\F_q$-algebra and $R_{\B}=\F_q[\A_0^{-1}\circ f_1, \ldots, \A_0^{-1}\circ f_m]$.
\item There exist unique nonnegative integers $t=t(\A)$ and $h=h(\A)$ and an element $\A'\in G^n$ such that $t$ entries of $\A'$ are of $t$-type, $h$ are of $h$-type and the $n-h-t$ remaining equal to the identity matrix with the additional property that $R_{\A}$ and $R_{\A'}$ are isomorphic, via the isomorphism induced by an element $\A_0\in G^n$.
\end{enumerate}
\end{theorem}

\begin{proof}
\begin{enumerate}[a)]
\item  Notice that $\A_0^{-1}\A\A_0=\B$. Hence, for any $f\in \mathcal A_n$, $\B\circ f=f$ if and only if $\A\circ (\A_0\circ f)=\A_0\circ f$, i.e., $\A_0\circ f\in R_{\A}$. In other words, $R_{\B}$ is the homomorphic image of $R_{\A}$ by the $\F_q$-automorphism $\A_0^{-1}$ of $\mathcal A_n$. Hence, if $\varphi_{\A, \B}:R_{\A}\rightarrow R_{\B}$ is the restriction of this automorphism to $R_{\A}$, $\varphi_{\A, \B}$ is an $\F_q$-isomorphism.  Suppose that $R_{\A}=\F_q[f_1, \ldots, f_m]$, where $f_i\in \mathcal A_n$, and let $g\in R_{\B}$. In particular, $\varphi_{\A, \B}^{-1}(g)$ is in $R_{\A}$, hence it is a polynomial expression in terms of the elements $f_1, \ldots, f_m$. Therefore, $g=\varphi_{\A, \B}(\varphi_{\A, \B}^{-1}(g))$ is a polynomial expression in terms of $\varphi_{\A, \B}(f_1), \ldots, \varphi_{\A, \B}(f_m)$. In other words, $R_{\B}\subseteq \F_q[\varphi_{\A, \B}(f_1), \ldots, \varphi_{\A, \B}(f_m)]$. The inverse inclusion follows in a similar way. Notice that, from definition, each $f_i$ is in $R_{\A}$, hence $\varphi_{\A, \B}(f_i)=\A_0^{-1}\circ f_i$ for $1\le i\le m$.

\item As we have seen, any coordinate of $\A\in G$ is conjugated in $G$ to an element of $h$-type or $t$-type. Write $\A=(A_1, \ldots, A_n)$ and let $C_H$ (resp. $C_T$) be the sets of integers $i$ (resp. $j$) with $1\le i, j\le n$ such that the $i$-th (resp. $j$-th) coordinate of $\A\in G$ is conjugated in $G$ to an element of $h$-type (resp. $t$-type), and set $h=|C_H|$, $t=|C_T|$. Also, for each $i\in C_h\cup C_t$, let $B_i$ be the element of $G$ such that $B_iA_iB_i^{-1}$ is of $t$-type or $h$-type and $B_i=I$ for $i\not\in C_h\cup C_t$. If we set $\A_0=(B_1, \ldots, B_n)$, notice that the element $\A'=\A_0\A\A_0^{-1}\in G^n$ is such that $t$ entries of $\A'$ are of $t$-type, $h$ are of $h$-type and the $n-h-t$ remaining equal to the identity matrix. Now, the result follows from the previous item. The uniqueness of the nonnegative integers $h$ and $t$ follows from the fact that the sets $C_H$ and $C_T$ are unique.
\end{enumerate}
\end{proof}

Theorem \ref{HT} shows that any element $\A \in G^n$ is conjugated to another element $\A'\in G^n$ in a reduced form (any coordinate is either of $h$-type, $t$-type or the identity matrix), such that the rings $R_{\A}$ and $R_{\A'}$ are isomorphic. We also note that, if we reorder the variables, no algebraic structure of the ring $R_{\A}$ is affected. From now, we assume that $\A\in G^n$ has the first coordinates as matrices of the $h$-type, the following ones of the $t$-type and the last ones equal  to the identity matrix.

\begin{definition}
Let $\A\in G^n$. For nonnegative integers $t$ and $h$ such that $t+h\le n$ we say that $\A$ is of the type $(h, t)$ if the first $h$ coordinates of $\A$ are of $h$-type, the following $t$ are of $t$-type and the $n-h-t$ remaining equal to the identity matrix.
\end{definition}

The type of $\A$, through the elements we are considering now, is well defined. We fix some notation on the coordinates of $h-$type of $\A$.

\begin{definition}
Let $\A$ be an element of $G^n$ of type $(h, t)$ and write $\A=(A_1, \ldots, A_n)$. For $h=0$, set $H(\A)=\emptyset$ and, for $h\ge 1$, set $H(\A)=\{a_1, \ldots, a_h\}\in \F_q^h$, where each $a_i$ is the first entry in the main diagonal of $A_i$ and $a_i\ne 1$ for $1\le i\le h$.
\end{definition}

It is clear that an element $\A$ of type $(h, t)$ is uniquely determined by $t$ and the set $H(\A)$.

\subsection{Translations and homotheties}

We start looking at the elements $\A$ of type $(0, t)$, i.e., maps consisting of translations $x_i\mapsto x_i+1$ for $1\le i\le t$, that fixes the remaining variables. In the univariate case we see that the set of invariant polynomials equals $\F_q[x^p-x]$. Let us see what happens in two variables: notice that $x^p-x$ and $y^p-y$ are polynomials invariant by the translations $x\mapsto x+1$ and $y\mapsto y+1$ and, if we consider these maps independently, i.e., if we look at the identity $$f(x+1, y)=f(x, y+1)=f(x, y),$$ one can show that the fixed-point subring is $\F_q[x^p-x, y^p-y]$. However, we are considering a less restrictive identity:
$$f(x+1, y+1)=f(x, y),$$
and in this case the polynomial $f(x, y)=x-y$ appears as an invariant element. Is not hard to see that $x-y$ does not belong to $\F_q[x^p-x, y^p-y]$. Is there any other exception? Well, notice that any polynomial $f(x, y)\in \F_q[x, y]$ can be written {\bf uniquely} as an univariate polynomial in $(x-y)$ with coefficients in $\F_q[y]$. In fact, $f(x, y)=g(x-y, y)$, where $g(x, y)=f(x+y, y)$. Hence $$f(x, y)=\sum_{i=0}^{m}(x-y)^iP_i(y),$$ and then $f(x+1, y+1)=f(x,y)$ if and only if $P_i(y+1)=P_i(y)$. In particular, from Theorem \ref{LR}, we know that each $P_i(y)$ is a polynomial expression in $t=y^p-y$. From this, one can show that the fixed-point subring is $\F_q[x-y, y^p-y]$. What about $x^p-x$? Well, notice that $x^p-x=(x-y)^p-(x-y)+(y^p-y)$, hence $x^p-x\in \F_q[x-y, y^p-y]$, as expected. 

As above, we will be frequently interested in writing an arbitrary polynomial $f\in \mathcal A_n$ as a finite sum of the form $\sum_{i\in B}g_ih_i$, where the variables appearing in $g_i\in \mathcal A_n$ are disjoint (or at least not contained) from the ones appearing in each $h_i\in \mathcal A_n$. This allows us to reduce our identities to well-known cases. We have the following:

\begin{lemma}\label{separation}
Let $m$ and $n$ be positive integers such that $m\le n$. Then any non zero $f\in \mathcal A_n$ can be written uniquely as 
$$f=\sum_{\beta \in B}\x^{\beta} P_{\beta},$$
where $B$ is a finite set (of distinct elements) and, for any $\beta \in B$, $\x^{\beta}$ is a monomial (or a constant) in $\F_q[x_1, \ldots, x_m]$ and each $P_{\beta}$ is a nonzero element of $\F_q[x_{m+1}\ldots, x_n]$ (which is $\F_q$ for $m=n$).
\end{lemma}

\begin{proof}
Since the variables are independent, $\mathcal A_n$ can be viewed as the ring of polynomials in the variables $x_1, \ldots, x_m$ with coefficients in $R=\F_q[x_{m+1}\ldots, x_n]$ and the result follows.\end{proof}

It is straightforward to check that, in Lemma \ref{separation}, we can replace $x_1, \ldots, x_m$ and $x_{m+1}, \ldots, x_n$ by any partition of $\{x_1, \ldots, x_n\}$ into 2 sets. We now present a natural extension of the ideas that we have discussed for translations in $\F_q[x,y]$:
\begin{theorem}\label{translation}
Suppose that $\A$ is of type $(0, t)$, where $t\le n$ is a nonnegative integer. Then $R_{\A}=\mathcal A_n$ if $t=0$, $R_{\A}=\F_q[x_1^p-x_1, x_2, \ldots, x_n]$ if $t=1$ and, for $t\ge 2$,
$$R_{\A}=\F_q[x_1-x_2, \ldots, x_{t-1}-x_t, x_t^p-x_t, x_{t+1}, \ldots, x_{n}].$$ In particular, $R_{\A}$ is a finitely generated $\F_q$-algebra.
\end{theorem}

\begin{proof}
The case $t=0$ is straightforward since the only element of type $(0, 0)$ is the the identity of $G^n$. For $t=1$ we obtain the equation $f(x_1+1, x_2, \ldots, x_n)=f(x_1, \ldots, x_n)$. From Lemma \ref{separation}, any $f\in \mathcal A_n$ can be written uniquely as 
$$f=\sum_{\alpha \in B}\x^{\alpha}P_\alpha(x_1),$$
where $B$ is a finite set and $\x^{\alpha}$ is a monomial in the variables $x_2, \ldots, x_n$. In particular, we have $f\in R_{\A}$ if and only if 
$$f=\sum_{\alpha \in B}\x^{\alpha}P_\alpha(x_1+1),$$
that is, $P_{\alpha}(x_1+1)=P_{\alpha}(x)$. From Theorem \ref{LR}, we know that the last equality holds if and only if $P_{\alpha}(x_1)$ is a polynomial in $x_1^p-x_1$, and then $R_{\A}=\F_q[x_1^p-x_1, x_2, \ldots, x_n]$. Suppose now that $t\ge 2$; given $f\in \mathcal A_n$, notice that, from Lemma \ref{separation}, $g=f(x_1+x_2, x_2, \ldots, x_n)\in \mathcal A_n$ can be written uniquely as $g=\sum_{i=1}^{n}x_1^iP_i$, where $P_i\in \F_q[x_2, \ldots, x_n]$ and $P_n$ is non zero. Therefore, $f=g(x_1-x_2, x_2, \ldots, x_n)$ can be written uniquely as
$$f=\sum_{i=1}^{n}(x_1-x_2)^iQ_i,$$
where $Q_i=P_i\in \F_q[x_2, \ldots, x_n]$. In particular, $f\in R_{\A}$ if and only if
$$f=\sum_{i=1}^{n}(x_1-x_2)^i\cdot \A\circ Q_i.$$
From the uniqueness of the polynomials $Q_i$ we have that $\A\circ f=f$ if and only if $\A\circ Q_i=Q_i$, where each $Q_i$ is in $\F_q[x_2, \ldots, x_n]$. In other words, $R_{\A}=L[x_1-x_2]$, where $L$ is the fixed-point subring of $\F_q[x_2, \ldots, x_n]$ by $\A$. We follow in the same way for the ring $L$. After $t-1$ iteration of this process, we obtain $R_{\A}=L_0[x_1-x_2, x_2-x_3, \ldots, x_{t-1}-x_t]$, where $L_0$ is the fixed-point subring of $\F_q[x_t, \ldots, x_n]$ by $\A$. Since $\A$ maps $x_t$ to $x_t+1$ and fixes $x_i$ for $t<i\le n$, we are back to the case $t=1$ (now with $n-t+1$ variables) and so $L_0=\F_q[x_t^p-x_t, x_{t+1}, \ldots, x_n]$. This finishes the proof.
\end{proof}

We introduce an useful notation:

\begin{definition}
\begin{enumerate}[(i)]
\item For any nonnegative integers $h, t$ such that $h+t\le n$, let $L(h, 0)=\emptyset$, $L(h, 1)=\{x_{h+1}^p-x_{h+1}\}$ and, for $t\ge 2$, $$L(h, t)=\{x_{h+1}-x_{h+2}, \ldots, x_{h+t-1}-x_{h+t}, x_{h+t}^p-x_{h+t}\}.$$
\item For any nonnegative integer $d\le n$, set $V_n=\emptyset$ and, for $d\le n-1$, $V_{d}=\{x_{d+1}, \ldots, x_{n}\}$.
\end{enumerate}
\end{definition}

From definition, Theorem \ref{translation} says that if $\A$ is of type $(0, t)$, the set $L(0, t)\cup V_{t}$ is a set of generators for $R_{\A}$ as an $\F_q$-algebra. We have the following ``translated'' version of Theorem \ref{translation}.

\begin{corollary}\label{trans}
Let $\Psi(h, t)$ be the $\F_q$-automorphism of $\mathcal A_n$, that maps $x_i$ to $x_i+1$ for $h+1\le i\le h+t$, where $h$ and $t$ are nonnegative integers such that $h+t\le n$. Let $f$ be a polynomial in $\F_{q}[x_{h+1}, \ldots, x_n]$.  Then $f$ is invariant by $\Psi(h, t)$ if and only if $f$ is a polynomial expression in terms of the elements of $L(h, t)\cup V_{h+t}$, i.e., the fixed point subring of $\F_{q}[x_{h+1}, \ldots, x_n]$ by $\Psi(h, t)$ coincides with the $\F_q$-algebra generated by $L(h, t)\cup V_{h+t}$.
\end{corollary}
\begin{proof}
Just notice that $\Psi(h,t )$, when restricted to $\F_{q}[x_{h+1}, \ldots, x_n]$, coincides with the automorphism induced by the element of type $(0, t)$ in $G^{n-h}$. We write $y_i=x_{h+i}$ for $1\le i\le n-h$ and the result follows from Theorem \ref{translation}.

\end{proof}

We now look at the elements $\A$ of type $(h, 0)$, i.e., maps consisting of homotheties $x_i\mapsto a_ix_i$ for $1\le i\le h$, that fixes the remaining variables. We have the following:

\begin{proposition}\label{homothety}
Suppose that $\A$ is of type $(h, 0)$, where $h$ is a nonnegative integer and, for $h\ge 1$, set $H(\A)=\{a_1, \ldots, a_h\}$ and $d_i=\ord(a_i)$. For $h\ge 1$, let $C_{\A}\in \N^{h}$ be the set of all vectors $(b_1, \ldots, b_h)\in \N^h$ such that $b_i\le d_i$, at least one $b_i$ is nonzero and \begin{equation}\label{product-one}a_1^{b_1}\cdots a_h^{b_h}=1.\end{equation}
For each $b\in C_{\A}$, $b=(b_1, \ldots, b_h)$ we associate the monomial $\y^{b}:=x_1^{b_1}\ldots x_h^{b_h}$. Let $M_{\A}:=\{\y^b\,|\, b\in C_{\A}\}$ and $h_{\A}:=|M_{\A}|$. Then $R_{\A}=\mathcal A_n$ if $h=0$ and, for $h\ge 1$, $$R_{\A}=\F_q[y_1, \ldots, y_{h_{\A}}, x_{h+1}, \ldots, x_{n}],$$ where $y_i$ runs through the distinct elements of $M_{\A}$. In particular, $R_{\A}$ is finitely generated as an $\F_q$-algebra.
\end{proposition}

\begin{proof}
The case $t=0$ is straightforward since the only element of type $(0, 0)$ is the the identity of $G^n$. Suppose that $h>0$. From Lemma \ref{separation}, we know that any $f\in \mathcal A_n$ can be written uniquely as 
$$f=\sum_{\alpha \in B}\x^{\alpha}P_{\alpha},$$
where $B$ is a finite set, each $\x^{\alpha}$ is a monomial in the variables $x_1, \ldots, x_h$ and each $P_{\alpha}$ is a nonzero element of $\F_q[x_{h+1}, \ldots, x_{n}]$. Notice that
$$\A\circ f=\sum_{\alpha \in B}\x^{\alpha}(a_{\alpha}P_{\alpha}),$$
where, for each $\alpha=(c_1, \ldots, c_n)\in \N^n$, $a_{\alpha}$ is defined as the product $a_1^{c_1}\ldots a_h^{c_h}$. If $f\in R_{\A}$, then $\A\circ f=f$ and, from the uniqueness of the polynomials $P_{\alpha}$, it follows that $a_{\alpha}P_{\alpha}=P_{\alpha}$ for any $\alpha\in B$. In particular, since $P_{\alpha}\ne 0$, we get $a_{\alpha}=1$, that is, $$a_1^{c_1}\cdots a_h^{c_h}=1.$$ 
Writing $c_i=d_iQ_i+r_i$, where $d_i=\ord(a_i)$ and $0\le r_i<d_i$ for $1\le i\le h$, we see that the last equality implies that $a_1^{r_1}\ldots a_h^{r_h}=1$, i.e., $(r_1, \ldots, r_h)$ is either the zero vector or belongs to $C_{\A}$. In other words, $\x^{\alpha}=(x_1^{d_1})^{Q_1}\ldots (x_h^{d_h})^{Q_h}\cdot \y^{b}$, where $\y^b$ is either $1\in \F_q$ or an element of $M_{\A}$. But notice that, since $a_i^{d_i}=1$, each $x_i^{d_i}$ is in $M_{\A}$. Hence $f=\sum_{\alpha \in B}\x^{\alpha}P_{\alpha}$ is such that each $\x^{\alpha}$ is a finite product of elements in $M_{\A}$ (or equal to $1\in \F_q$) and, in particular, $R_{\A}\subseteq \F_q[y_1, \ldots, y_{h_{\A}}, x_{h+1}, \ldots, x_{n}]$, where $y_i$ runs through the distinct elements of $M_{\A}$. For the reverse inclusion $R_{\A}\supseteq \F_q[y_1, \ldots, y_{h_{\A}}, x_{h+1}, \ldots, x_{n}]$, notice that each monomial $y_i$ satisfies $\A\circ y_i=y_i$ and $\A$ trivially fixes the variables $x_{h+1}, \ldots, x_n$. Thus $R_{\A}=\F_q[y_1, \ldots, y_{h_{\A}}, x_{h+1}, \ldots, x_{n}]$ and we conclude the proof.
\end{proof}

For instance, suppose that $q$ is odd and, for $f\in \F_q[x, y, z]$, consider the identity $$f(x, y, z)=f(-x, -y, z).$$ In other words, $\A\circ f=f$, where $\A\in G^n$ is of type $(2, 0)$ and $H(\A)=\{-1, -1\}$. From Proposition \ref{homothety} we get $M_{\A}=\{x^2, y^2, xy, x^2y^2\}$ and $$R_{\A}=\F_q[x^2, y^2, xy, x^2y^2, z].$$ However, the element $x^2y^2$ is already in $\F_q[x^2, y^2, xy, z]$, since $x^2y^2=x^2\cdot y^2$ or even $x^2y^2=(xy)^2$. We then may write $R_{\A}=\F_q[x^2, y^2, xy, z]$. We introduce a subset of $M_{\A}$ that removes these {\it redundant} elements. 

\begin{definition}
Suppose that $\A$ is of type $(h, 0)$, where $h$ is a non negative integer. Let $M_{\A}^*=\emptyset$ if $h=0$ and, for $h\ge 1$, $M^*_{\A}$ is the subset of $M_{\A}$ comprising the monomials $\x^{\alpha}$ of $M_{\A}$ that are not divisible by any element of $M_{\A}\setminus \{\x^{\alpha}\}$, where $M_{\A}$ is as in Theorem \ref{homothety}. We set $N(\A)=|M_{\A}^*|-h$.
\end{definition}
From definition, if $\A$ is of type $(h, 0)$, where $h\ge 1$ and $H(\A)=\{a_1, \ldots, a_h\}$, then $\{x_1^{d_1}, \ldots, x_{h}^{d_h}\}\subseteq M_{\A}^*$, where $d_i=\ord(a_i)$; one can verify that any other element of $M_{\A}^*$ is a ``mixed'' monomial. We always have the bound $N(\A)\ge 0$ and, in fact, $N(\A)$ counts the number of mixed monomials appearing in $M_{\A}^*$. In the previous example, we have $M_{\A}^*=\{x^2, y^2, xy\}$ and $N(\A)=1$. 

Notice that any element of $M_{\A}\supseteq M_{\A}^*$ is a finite product of elements in $M_{\A}^*$; in fact, pick an element $\x^{\alpha}\in M_{\A}$. If $m_0:=\x^{\alpha}\in M_{\A}^*$, we are done. Otherwise, $\x^{\alpha}$ is divisible by some $\x^{\beta}\in M_{\A}^*\setminus \{\x^{\alpha}\}$. From the definition of $M_{\A}$, it follows that $m_1:={\x^{\alpha}}/{\x^{\beta}}$ is also an element of $M_{\A}$. We then proceed in the same way for the element $m_1={\x^{\alpha}}/{\x^{\beta}}$. This give us a sequence of monomials $\{m_0, m_1, \ldots \}$ such that the corresponding sequence of multidegrees $\{ \mathrm{mdeg}(m_0), \mathrm{mdeg}(m_1), \ldots \}$ is decreasing (with respect to the graded lex order). Therefore, after a finite number of steps we will arrive in an element $m_j\in M_{\A}^*$ and this process gives us $\x^{\alpha}$ as a finite product of elements in $M_{\A}^*$. In particular, if $\A$ is of type $(h, 0)$, then $M_{\A}^*\cup V_{h}$ and $M_{\A}\cup V_h$ generate the same $\F_q$-algebra, i.e., $M_{\A}^*\cup V_{h}$ generates
$R_{\A}$ as an $\F_q$-algebra. We finish this section showing that $M_{\A}^*$ is minimal in some sense.

\begin{lemma}\label{monomial-minimal}
For any $\x^{\alpha}\in M_{\A}^*$, $\x^{\alpha}$ cannot be written as a polynomial expression in terms of the elements in $M_{\A}^*\setminus \{\x^{\alpha}\}$.
\end{lemma}

\begin{proof}
Suppose that there is an element $\x^{\alpha}\in M_{\A}^*$ with this property; such a polynomial expression in terms of the elements in $M_{\A}^*\setminus \{\x^{\alpha}\}$ has constant term equals zero (we can see this, for instance, evaluating at the point $(0, \ldots, 0)\in \F_q^n$). In particular, $\x^{\alpha}$ belongs to the monomial ideal generated by the set $M_{\A}^*\setminus \{\x^{\alpha}\}$ in $\mathcal A_n$. But it is well known that, a monomial belongs to the monomial ideal $I\subset \mathcal A_n$ generated by a set $C$ if and only if the monomial itself is divisible by some element in $C$. But, from definition, $\x^{\alpha}$ is not divisible by any element of $M_{\A}\setminus \{ \x^{\alpha}\}\supseteq M_{\A}^*\setminus \{ \x^{\alpha}\}$ and we get a contradiction.

\end{proof}

\section{The structure of the fixed-point subring $R_{\A}$}
In the previous section we characterize the fixed-point subring $R_{\A}$ in the case that $\A$ is of type $(0, t)$ or $(h, 0)$. We show how this extends to the general case.  

\begin{proposition}\label{fixed-ring}
Suppose that $\A\in G^n$ is of type $(h, t)$. There exist unique elements $\A_1$ and $\A_2$ with the following properties: 
\begin{enumerate}[(i)]
\item $\A_1$ if of type $(h, 0)$.
\item The first $h$ and the last $n-h-t$ coordinates of $\A_2$ are the identity matrix and the remaining $t$ (in the middle) are elements of $t-$type.
\item $\A=\A_1\cdot \A_2$. 
\end{enumerate}
Moreover, $R_{\A}=R_{\A_1}\cap R_{\A_2}$ and, in particular, $R_{\A}$ is the $\F_q$-algebra generated by the elements of $M_{\A_1}^*\cup L(h, t)\cup V_{h+t}$. 
\end{proposition}

\begin{proof}
Write $\A=(A_1, \ldots, A_n)$ and set $\A_1=(A_1, \ldots, A_h, I, \ldots, I)\in G^n$, where each $A_i$ is of $h$-type. Given $\A$ of type $(h, t)$, such an $\A_1$ is unique. We notice that $\A_2=\A_1^{-1}\A$ has the required properties. 

It turns out that the elements $\A_1$ and $\A_2$ commute in $G^n$ and $D_1=\ord(\A_1)$, $D_2=\ord(\A_2)$ divide $q-1$ and $p$, respectively. In particular, since $p$ and $q-1$ are relatively prime, then so are $D_1$ and $D_2$. Let $r$ and $s$ be positive integers such that $rD_1\equiv 1\pmod {D_2}$ and $sD_2\equiv 1 \pmod {D_1}$. If $f\in \mathcal A_n$ is such that $\A\circ f=f$, then $$\A^{rD_1}\circ f=\A^{sD_2}\circ f=f.$$ Since $\A_1$ and $\A_2$ commute in $G^n$, it follows that $\A^{rD_1}=\A_1^{rD_1}\cdot \A_2^{rD_1}=\A_2$ and $\A^{sD_2}=\A_1^{sD_2}\cdot \A_2^{sD_2}=\A_1$. In particular, $\A_1\circ f=\A_2\circ f=f$. Therefore $R_{\A}\subseteq R_{\A_1}\cap R_{\A_2}$. The reverse inclusion is trivial and then $R_{\A}=R_{\A_1}\cap R_{\A_2}$. 

Let $R$ be the $\F_q$-algebra generated by the elements of $M_{\A_1}^*\cup L(h, t)\cup V_{h+t}$. From Theorems \ref{translation}, \ref{homothety} and Corollary \ref{trans}, we see that any element $f\in R$ satisfies $\A_1\circ f=\A_2\circ f=f$ and then $R\subseteq R_{\A_1}\cap R_{\A_2}=R_{\A}$. Conversely, suppose that $f\in R_{\A}=R_{\A_1}\cap R_{\A_2}$. From Lemma \ref{separation}, $f$ can be written uniquely as
$$f=\sum_{\alpha}\x^{\alpha}P_{\alpha}$$
where $B\subset \N^{h}$ is a finite set, each $\x^{\alpha}$ is a monomial in $\F_q[x_1, \ldots, x_h]$ and $P_{\alpha}$ is a nonzero polynomial in $\F_{q}[x_{h+1}, \ldots, x_n]$. Since $\A_2$ fixes each monomial $\x^{\alpha}, \alpha \in B$ and $\A_2\circ f=f$, we obtain $\A_2 \circ P_{\alpha}=P_{\alpha}$ and then, from Corollary \ref{trans}, we see that each $P_{\alpha}$ is a polynomial expression in terms of the elements in $L(h, t)\cup V_{h+t}$. Also, since $\A_1$ fixes each polynomial $P_{\alpha}, \alpha \in B$ and $\A_1\circ f=f$, we obtain $(\A_1 \circ \x^{\alpha})\cdot P_{\alpha}=\x^{\alpha}\cdot P_{\alpha}$ and then, since $P_{\alpha}$ is nonzero, we conclude that $\A_1 \circ \x^{\alpha}=\x^{\alpha}$. Therefore, from Theorem \ref{homothety}, we have that each $\x^{\alpha}$ is a polynomial expression in terms of the elements in $M_{\A_1}^*$. In particular, $f$ must be a polynomial expression in terms of the elements of $M_{\A_1}^*\cup L(h,t )\cup V_{h+t}$, i.e., $f\in R$. Thus $R=R_{\A}$, as desired.
\end{proof}
From now, if $\A$ is an element of type $(h, t)$, we say that the identity $\A=\A_1\A_2$ as in Proposition \ref{fixed-ring} is the {\it canonical decomposition} of $\A$. 

For instance, consider the element $\A\in G^5$ of type $(2, 2)$ with canonical decomposition $\A=\A_1\A_2$, where $H(\A)=H(\A_1)=\{-1, -1\}$. The ring $R_{\A}$ comprises the polynomials $f\in A_5$ such that
$$f(x_1, \ldots, x_5)=f(-x_1, -x_2, x_3+1, x_4+1, x_5).$$
In this case, Proposition \ref{fixed-ring} says that $R_{\A}=\F_q[x_1^2, x_2^2, x_1x_2, x_3-x_4, x_4^p-x_4, x_5]$. 

We ask if $M_{\A_1}^*\cup L(h, t)\cup V_{h+t}$ contains redundant elements, i.e., can we remove some elements and still generate the ring $R_{\A}$? This lead us to introduce the following:

\begin{definition}
Suppose that $R\in \mathcal A_n$ is a finitely generated $\F_q$-algebra and let $S$ be a set of generators for $R$. We say that $S$ is a minimal generating set for $R$ if there is no proper subset $S'\subset S$ such that $S'$ generates $R$. 
\end{definition}

In other words, minimal generating sets $S$ are those ones with the property that no element $E$ of $S$ can be written as a polynomial expression in terms of the elements in $S\setminus \{E\}$. We will prove that the set of generators for $R_{\A}$ given in Proposition \ref{fixed-ring} is minimal, but first we explore the {\it algebraic independence} on the set $M_{\A_1}^*\cup L(h, t)\cup V_{h+t}$ (which is, somehow, stronger than the concept of redundant elements).

\subsection{Algebraic independence in positive characteristic}
If $K$ is an arbitrary field, given polynomials $f_1, \ldots, f_m$ in $K[x_1, \ldots, x_n]$, we say that $f_1, \ldots, f_m$ are {\it algebraically independent} if there is no nonzero polynomial $P\in K[y_1, \ldots, y_m]$ such that $P(f_1, \ldots, f_m)$ is identically zero in $K[x_1, \ldots, x_n]$. For instance, the polynomials $x^2$ and $y$ are algebraically independent over $K[x, y]$, but $x^2+y^2$ and $x+y$ are not algebraically independent over $K[x, y]$, where $K$ is any field of characteristic two; $x^2+y^2+(x+y)^2=0$. From definition, if $f_1, \ldots, f_m$ are algebraically independent, any subset of such polynomials has the same property.

Given polynomials $f_1, \ldots, f_n$ in $K[x_1, \ldots, x_n]$, we define their {\it Jacobian} as the polynomial
$$\det (J(f_1, \ldots, f_n)),$$
where $J(f_1, \ldots, f_n)$ is the $n\times n$ matrix with entries $a_{ij}=\frac{\partial f_i}{\partial x_j}$; here, $\frac{\partial f_i}{\partial x_j}$ denotes the {\it partial derivative} of $f_i$ with respect to $x_j$. The well known {\it Jacobian Criterion} says that, over characteristic zero, a set of $n$ polynomials in $K[x_1, \ldots, x_n]$ is algebraically independent if and only if their Jacobian is nonzero. This may fail in positive characteristic; notice that $x^p$ and $y^p$ are algebraically independent over $\F_p[x, y]$, but $\det (J(x^p, y^p))=0$. However, we have at least one direction of this result:

\begin{theorem}[Jacobian Criterion - weak version] Suppose that $f_1, \ldots, f_n$ is a set of polynomials in $\F_q[x_1, \ldots, x_n]$ such that their Jacobian is nonzero. Then $f_1, \ldots, f_n$ are algebraically independent.
\end{theorem}

In particular, we obtain the following:

\begin{corollary}\label{AI}
For any nonnegative integers $h$ and $t$ such that $h+t\le n$ and any sequence $d_1, \ldots, d_h$ (which is empty for $h=0$) of divisors of $q-1$, the $n$ elements of $\{x_1^{d_1}, \ldots, x_h^{d_h}\}\cup L(h, t)\cup V_{h+t}$ are algebraically independent.
\end{corollary}

\begin{proof}
By a direct verification we see that the Jacobian of the $n$ polynomials in the set $\{x_1^{d_1}, \ldots, x_h^{d_h}\}\cup L(h, t)\cup V_{h+t}$ equals $1$ if $h=t=0$ and $$\varepsilon(t)\cdot (d_1\cdots d_h)\cdot (x_1^{d_1-1}\cdots x_h^{d_h-1}),$$ if $h\ne 0$, where
$\varepsilon(t)=1$ for $t=0$ and $\varepsilon(t)=-1$ for $t\ne 0$. Since each $d_i$ is a divisor of $q-1$ (which is prime to the characteristic $p$), this Jacobian is never zero and the result follows from the (weak) Jacobian Criterion for $\F_q$.
\end{proof}

We are now able to prove the minimality of $M_{\A_1}^*\cup L(h, t)\cup V_{h+t}$ as a set of generators for $R_{\A}$.

\begin{proposition}
Let $\A\in G^n$ be an element of type $(h, t)$ and $\A=\A_1\A_2$ its canonical decomposition. Then $M_{\A_1}^*\cup L(h, t)\cup V_{h+t}$ is a minimal generating set for $R_{\A}$. 
\end{proposition}

\begin{proof}
We already know that this set is a generator. To prove the minimality of such set, let $I_T$ be the ideal generated by $L(h, t)\cup V_{h+t}$ over the ring $R_{\A}$  (with the convention that $I_T$ is the zero ideal if the corresponding set is empty). We first show that no element of $M_{\A_1}^*$ is redundant. For this, suppose that an element $\x^{\alpha}\in M_{\A_1}^*$ is a polynomial expression in terms of the elements in $M_{\A_1}^*\cup L(h, t)\cup V_{h+t}\setminus\{\x^{\alpha}\}$. Looking at the quotient $R_{\A}/I_T$, this yields an equality
$$x^{\alpha}\equiv P_{\alpha} \pmod{I_T},$$
where $P_{\alpha}$ is a polynomial expression in terms of the elements in $M_{\A_1}^*\setminus\{\x^{\alpha}\}$. In other words, $\x^{\alpha}-P_{\alpha}$ is an element of $I_T$. One can see that this implies $\x^{\alpha}-P_\alpha=0$, a contradiction with Lemma \ref{monomial-minimal}. In the same way (taking $I_H$ as the ideal generated by $M_{\A_1}^*\cup V_{h+t}$ over $R_{\A}$) we see that if there is a redundant element $T$ in $L(h, t)$, then such a $T$ can be written as a polynomial expression in terms of the elements of $L(h, t)\setminus \{T\}$. But this yields a nonzero polynomial $P\in \F_q[y_1, \ldots, y_{t}]$ such that
$$P(T_1, \ldots, T_t)=0,$$
where $T_i$ runs through the elements of $L(h, t)$, which is impossible since Lemma \ref{AI} ensures that these elements are algebraically independent. Finally, since the variable $x_i$ does not appear in the set $M_{\A_1}^*\cup L(h, t)\cup V_{h+t}\setminus\{x_i\}$ for $h+t<i\le n$, it follows that no element $x_i\in V_{h+t}$ can be written as a polynomial expression in terms of the elements of $M_{\A_1}^*\cup L(h, t)\cup V_{h+t}\setminus\{x_i\}$. This finishes the proof.
\end{proof}

For an element $\A$ of type $(h, t)$ with canonical decomposition $\A=\A_1\A_2$ we say that $S_{\A}:=M_{\A_1}^*\cup L(h, t)\cup V_{h+t}$ is the {\it canonical generating set} for $R_{\A}$. We also set $N_{\A}=|S_{\A}|$ which is, from definition, equal to $N(\A_1)+n$.

\subsection{Free algebras}
Given a field $K$ and a finitely generated $K$-algebra $R\subseteq K[x_1, \ldots, x_n]$, we say that $R$ is {\it free} if $R$ can be generated by a finite set $f_1, \ldots, f_m\in K[x_1, \ldots, x_n]$ comprising algebraically independent elements. In other words, $R$ (as a ring) is isomorphic to the ring of $m$ variables over $K$, for some $m$.

As follows, we have a simple criterion for $R_{\A}$ to be a free $\F_q$-algebra.
\begin{theorem}\label{free}
Let $\A\in G^n$ be an element of type $(h, t)$ and $\A=\A_1\A_2$ its canonical decomposition. Write $H(\A)=H(\A_1)=\{a_1, \ldots, a_h\}$ for $h\ge 1$ and $d_i=\ord(a_i)>1$ for $1\le i\le h$. The following are equivalent:
\begin{enumerate}[(i)]
\item $R_{\A}$ is free.
\item $h=0, 1$ or $h>1$ and the numbers $d_i$ are pairwise coprime.
\item $N(\A_1)=0$.
\item $R_{\A}$ is isomorphic to $\mathcal A_n$.
\end{enumerate}
\end{theorem}
\begin{proof}
\textbf{(i) $\rightarrow$ (ii)}: it sufficient to prove that, if $h>1$ and there are two elements $d_i$ and $d_j$ not relatively prime, then $R_{\A}$ is not a free $\F_q$-algebra. Without loss of generality, suppose that $\gcd(d_1, d_2)=d>1$. If $R_{\A}$ were a free $\F_q$-algebra, then it would be isomorphic to the ring $K[y_1, \ldots, y_m]$ of $m$ variables over $K$ for some $m$, which is always an {\it Unique Factorization Domain}. As we will see, the ring $R_{\A}$ does not have this property. 

For a given primitive element $\theta\in \F_{q}^*$, notice that $a_1=\theta^{\frac{q-1}{d_1}r_1}$ and $a_2=\theta^{\frac{q-1}{d_2}r_2}$ for some positive integers $r_1\le d_1$ and $r_2\le d_2$ such that $\gcd(r_1, d_1)=\gcd(r_2, d_2)=1$. In particular, since $d$ divides $d_1$, $d$ is coprime with $r_1$ and so there exists a nonnegative integer $j\le d-1$ such that $jr_1\equiv -r_2\pmod d$; since $d$ coprime with $r_2$, $j\ne 0$. Notice that
$$a_1^{\frac{j}{d}d_1}a_2^{\frac{d_2}{d}}=\theta^{\frac{(q-1)(jr_1+r_2)}{d}}=1,$$
since $jr_1+r_2$ is divisible by $d$. In particular, $(\frac{jd_1}{d}, \frac{d_2}{d}, 0\ldots, 0)\in \N^h$ satisfies Eq.~\eqref{product-one}. Clearly $u_1:=\frac{jd_1}{d}<d_1$ and $u_2:=\frac{d_2}{d}<d_2$, and then $\y=x_1^{u_1}x_2^{u_2}$ belongs to $M_{\A_1}$. It follows from definition that this monomial is divisible by a monomial $\x:=x_1^{v_1}x_2^{v_2}\in M_{\A_1}^*$ with $v_i\le u_i<d_i$; in particular, $a_1^{v_1}a_2^{v_2}=1$ and, since $v_i<d_i$, it follows that $v_1, v_2>0$. 

Is not hard to see that $R_{\A}$, viewed as a ring, is an {\it Integral Domain} and any element of $M_{\A_1}^*$ is {\it irreducible} over $R_{\A}$. Notice that
$$\y^{d}=(x_1^{d_1})^j\cdot x_2^{d_2},$$
hence $\x=x_1^{v_1}x_2^{v_2}$ is an irreducible dividing $\y^{d}$ and, since $v_1, v_2>0$, $\x$ does not divide $$(x_1^{d_1})^j\quad \text{or}\quad x_2^{d_2}.$$ However, $x_1^{d_1}$ and  $x_2^{d_2}$ are in $R_{\A}$ and then $R_{\A}$ cannot be an Unique Factorization Domain.

\textbf{(ii) $\rightarrow$ (iii)}: For $h=0$ or $1$, $M_{\A_1}^*=\emptyset$ or $\{x_1^{d_1}\}$, respectively, and in both cases $N(\A_1)=0$. Let $h>1$ and suppose that the numbers $d_i$ are pairwise coprime. We are going to find explicitly the set $M_{\A_1}^*$: suppose that $(b_1, \ldots, b_h)\in \N^{h}$, where $b_i\le d_i$, at least one $b_i$ is nonzero and
$$a_1^{b_1}\cdots a_h^{b_h}=1.$$
Set $D=d_1\cdots d_h$ and $D_j=\frac{D}{d_j}$, for $1\le j\le h$. Raising powers $D_j$ in the previous equality we obtain
$$a_j^{b_jD_j}=1,$$
hence $d_j$ divides $b_jD_j$. In particular, since the numbers $d_i$ are pairwise coprime, it follows that $d_j$ and $D_j$ are coprime and then we conclude that $d_j$ divides $b_j$. Since $0\le b_j\le d_j$, it follows that, for each $1\le j\le h$, either $b_j=0$ or $b_j=d_j$. This shows that $M_{\A_1}^*=\{x_1^{d_1}, \ldots, x_h^{d_h}\}$ and then $$N(\A_1)=|M_{\A_1}^*|-h=h-h=0$$

\textbf{(iii) $\rightarrow$ (iv)}: If $N(\A_1)=0$, we know that $M_{\A_1}^*=\{x_1^{d_1}, \ldots, x_h^{d_h}\}$. From Proposition \ref{fixed-ring} and Corollary \ref{AI} it follows that $R_{\A}$ is an $\F_q$-algebra generated by $n$ algebraically independent elements in $\mathcal A_n$. In particular, $R_{\A}$ is isomorphic to $\mathcal A_n$.

\textbf{(iv)$\rightarrow$ (i)}: This follows directly by definition.
\end{proof}

In other words, Theorem \ref{free} above says that $R_{\A}$ is free if and only if $M_{\A_1}^*$ has no mixed monomials. For instance, if $q=2$, for any $n\ge 1$ and $\A\in G^n$, $\A$ has no elements of $h$-type as coordinates. In particular, the algebra $R_{\A}$ is always isomorphic to $\mathcal A_n$. 

We have a sharp upper bound on the number of coordinates of $h-$type in an element $\A$ such that $R_{\A}$ is free:

\begin{corollary}
Suppose that $q>2$ and let $\omega(q-1)$ be the number of distinct prime divisors of $q-1$. The following hold:
\begin{enumerate}[(i)]
\item If $\A\in G^n$ is an element of type $(h, t)$ with $h>\omega(q-1)$, then $R_{\A}$ is not free. 
\item For any nonnegative integers $h\le \omega(q-1)$ and $t$ such that $h+t\le n$, there is an element $\A$ of type $(h, t)$ such that $R_{\A}$ is free.
\end{enumerate}
\end{corollary}

\begin{proof}
\begin{enumerate}[(i)]
\item Let $\A$ be an element of type $(h, t)$ with $h>\omega(q-1)$ such that $H(\A)=\{a_1, \ldots, a_h\}$ and $d_i=\ord(a_i)>1$ for $1\le i\le h$. Since $h>\omega(q-1)$, from the {\it Pigeonhole Principle}, there exist two elements $d_i$ and $d_j$ that are divisible by some prime factor $r$ of $q-1$ and it follows from Theorem \ref{free} that $R_{\A}$ is not free. 
\item If $\omega(q-1)\le 1$, then $h\le 1$ and Theorem \ref{free} says that, for any element $\A$ of type $(h, t)$, the $\F_q$-algebra $R_{\A}$ is free. Suppose that $\omega(q-1)>1$, $2\le h\le \omega(q-1)$ and let $p_1, \ldots, p_h$ be distinct prime factors of $q-1$. For each $1\le i\le h$, let $\theta_i\in \F_q^*$ be an element such that $\ord(\theta_i)=p_i$. For each nonnegative integer $t$ with $h+t\le n$, consider $\A$ the element of type $(h, t)$ such that $H(\A)=\{\theta_1, \ldots, \theta_h\}$. Since the numbers $p_i$ are pairwise coprime, from Theorem \ref{free}, $R_{\A}$ is a free $\F_q$-algebra.
\end{enumerate}
\end{proof}

\section{Minimal product-one sequences in $\F_q^*$ and bounds for $N(\A_1)$}

In the previous section we give a characterization of the ring $R_{\A}$ as a finitely generated $\F_q$-algebra, finding explicitly a minimal generating set $S_{\A}$ for $R_{\A}$. How large can be the set $S_{\A}$? We have seen that $|S_{\A}|=n+N(\A_1)$ and actually $N(\A_1)=0$ when $\A$ is of type $(h, t)$ for $h=0, 1$ and some special cases of $h\ge 2$. Is then natural to ask what happens in the general case $h\ge 2$. In this section, we show that $N(\A_1)$ is related to the number of minimal solutions of Eq.~\eqref{product-one} and show how this can be translated to the study of minimal product-one sequences in $\F_q^*$. We start with some basic theory on product-one sequences.

\begin{definition}
Given a finite abelian group $H$ (written multiplicatively), a sequence of elements $(a_1, \ldots a_k)$ (not necessarily distinct) in $H$ is a product-one sequence if $a_1\cdots a_k=1$, where $1$ is the identity of $H$; the number $k$ is called the length of $(a_1, \ldots, a_k)$.  We say that the sequence $(a_1, \ldots, a_k)$ is a minimal product-one sequence if $a_1\cdots a_k=1$ and no subsequence of $(a_1, \ldots a_k)$ share the same property.
\end{definition}
Since we are working in abelian groups, we consider the sequences up to permutation of their elements. The so-called {\it Davenport constant} of $H$, denoted by $D(H)$, is the smallest positive integer $d$ such that any sequence in $H$ of $d$ elements contains a product-one subsequence. In other words, $D(H)$ is the maximal length of minimal product-one sequences in $H$. What are the sequences attaining this bound? In the case when $H$ is cyclic, things are well understood:

\begin{theorem}\label{zero-sum}
Suppose that $H=C_m$ is the cyclic group with $m$ elements. Then $D(C_m)=m$. Also, any minimal product-one sequence in $C_m$ of size $m$ is of the form $(g, \ldots, g)$ for some generator $g$ of $C_m$.
\end{theorem}

Recall that, for an element $\A$ of type $(h, 0)$ with $h\ge 1$ and $H(\A)=\{a_1, \cdots a_h\}$, the set $M_{\A}$ is defined as the set of monomials $\x=x_1^{b_1}\cdots x_h^{b_h}$ such that at least one $b_i$ is nonzero, $b_i\le \ord(a_i)$ and 
$$a_1^{b_1}\cdots a_h^{b_h}=1.$$
In particular, the monomial $\x\in M_{\A}$ can be associated to the following product-one sequence $a(\x)$ in the cyclic group $\F_q^*=C_{q-1}$: 
$$a(\x):=(\underbrace{a_1, \ldots, a_1}_{b_1\,\mathrm{times}}, \ldots, \underbrace{a_h, \ldots, a_h}_{b_h\,\mathrm{times}}).$$ 
Its length is $\sum_{j=1}^hb_j$.  We claim that, for each $\x \in M_{\A}^*$, its associated product-one sequence is minimal. In fact, if $a(\x)$ were not minimal, there would exist nonnegative integers $b_1', \cdots, b_h'$ such that $b_i'\le b_i$, at least one $b_i'$ is nonzero, at least one $b_j'$ is strictly smaller than the corresponding $b_j$ and
$$a_1^{b_1'}\cdots a_h^{b_h'}=1.$$
From definition, the monomial $\y=x_1^{b_1'}\cdots x_h^{b_h'}$ is in $M_{\A}$ and divides $\x$, a contradiction since $\x\in M_{\A}^*$. Based on this observation and Theorem \ref{zero-sum}, we can give a sharp upper bound for the numbers $N(\A_1)$:

\begin{theorem}\label{upper-bound}
Let $\A\in G^n$ be an element of type $(h, t)$ and $\A=\A_1\A_2$ its canonical decomposition, where $h\ge 2$. Write $H(\A)=H(\A_1)=\{a_1, \ldots, a_h\}$ and $d_i=\ord(a_i)>1$ for $1\le i\le h$. Also, let $\ell(\A)$ be the least common multiple of the numbers $d_1, \ldots, d_h$. The following hold:
\begin{enumerate}[a)]
\item $N(\A_1)\le \binom{\ell(\A)+h-1}{h-1}-h$ and, in particular, $|S_{\A}|\le \binom{\ell(\A)+h-1}{h-1}+n-h$.
\item $N(\A_1)=\binom{\ell(\A)+h-1}{h-1}-h$ if and only if $H(\A)=H(\A_1)=\{\theta, \theta, \ldots, \theta\}$, where $\theta$ is an element of order $\ell(\A)$ in $\F_q^*$.
\end{enumerate}
\end{theorem}

\begin{proof}
\begin{enumerate}[a)]
\item Since each $d_i$ divides $q-1$, it follows that $\ell(\A)$ divides $q-1$. Let $C_{\ell(\A)}\subseteq \F_{q}^*$ be the cyclic group of order $\ell(\A)$. In particular, since $a_i^{\ell(\A)}=1$ for any $1\le i\le h$, each $a_i$ is in $C_{\ell(\A)}$. We have seen that any element $x_1^{b_1}\cdots x_h^{b_h}\in M_{\A_1}^*$ can be associated to a minimal product-one sequence in $\F_{q}^*$ having length $\sum_{j=1}^hb_j$. Actually, since each $a_i$ is in $C_{\ell(\A)}$, such a sequence is in $C_{\ell(\A)}$. From Theorem \ref{zero-sum}, it follows that its length $\sum_{j=1}^hb_j$ is at most $\ell(\A)$. In particular, any monomial $x_1^{b_1}\cdots x_h^{b_h}\in M_{\A_1}^*$ is such that $\sum_{j=1}^hb_j\le \ell(\A)$. If $\mathcal M_h(d)$ denotes the set of all monomials $x_1^{r_1}\cdots x_h^{r_h}$ such that $\sum_{i=1}^hr_i=d$, we define the following map:
$$
 \begin{array}{rccl}
  \Lambda_h : & M_{\A_1}^* &\longrightarrow& \mathcal M_h(\ell(\A))\\
  &x_1^{b_1}\cdots x_h^{b_h} &\longmapsto& x_1^{b_1}\cdot x_2^{b_2}\cdots x_h^{b_h+\ell(\A)-(b_1+\cdots +b_h)}.
\end{array}
$$
Clearly, $\Lambda_h$ is well defined. We claim that $\Lambda_h$ is one-to-one. In fact, if there are two distinct elements $\x_1= x_1^{b_1}\cdots x_h^{b_h}$ and $\x_1'=x_1^{b_1'}\cdots x_h^{b_h'}$ in $M_{\A_1}^*$ such that $\Lambda_h(\x_1)=\Lambda_h(\x_1')$, we have $b_i=b_i'$ for $1\le i\le h-1$ and then, since the elements are distinct, it follows that $b_h\ne b_h'$. For instance, suppose $b_h>b_h'$, hence $\x_1$ is divisible by $\x_1'$, a contradiction with the definition of $M_{\A_1}^*$. In particular, we have shown that $|M_{\A_1}^*|\le |\mathcal M_h(\ell(\A))|$. A simple calculation yields $|\mathcal M_h(\ell(\A))|=\binom{\ell(\A)+h-1}{h-1}$. Thus,
$$N(\A_1)=|M_{\A_1}|-h\le \binom{\ell(\A)+h-1}{h-1}-h.$$

\item Suppose that $N(\A_1)=\binom{\ell(\A)+h-1}{h-1}-h$. In particular, $$|M_{\A_1}^*|=N(\A_1)+h=\binom{\ell(\A)+h-1}{h-1}=|\mathcal M_h(\ell(\A))|,$$ 
and then the map $\Lambda_h$ defined above is an one-to-one correspondence. 
In particular, the elements $x_i^{\ell(\A)}$ are in the image of $M_{\A_1}^*$ by $\Lambda_h$ and this implies $x_i^{\ell(\A)}\in M_{\A_1}^*$ for any $1\le i <h$. From the definition of $M_{\A_1}^*$, it follows that $d_i=\ell(\A)$ for $1\le i<h$. Similarly, if we define the map $\Lambda_h'$ as 
$$
 \begin{array}{rccl}
  \Lambda_h' : & M_{\A_1}^* &\longrightarrow& \mathcal M_h(\ell(\A))\\
  &x_1^{b_1}\cdots x_h^{b_h} &\longmapsto& x_1^{b_1+\ell(\A)-(b_1+\cdots +b_h)}\cdot x_2^{b_2}\cdots x_h^{b_h},
\end{array}
$$
one can show that $\Lambda_h'$ must be an one-to-one correspondence and in the same way we obtain $d_i=\ell(\A)$ for any $1<i\le h$. 
Thus $d_i=\ell(\A)$ for any $1\le i\le h$, i.e., the multiplicative orders of the elements $a_i$ are the same. We split into cases:
\begin{enumerate}[]
\item {\bf Case 1.} $h=2$. Since $\Lambda_h$ is onto, it follows that $x_1^{k}x_2^{\ell(\A)-k}$ is in the image of $M_{\A_1}^*$ by $\Lambda_h$, for any $1\le k<\ell(\A)$. But the pre-image of such element is $x_1^kx_2^{s(k)} \in M_{\A_1}^*$ for some positive integer $1\le s(k)<\ell(\A)$. From the definition of $M_{\A_1}^*$, $s(k)\ne s(k')$ if $k\ne k'$. Also, $x_1x_2^{s(1)}$ cannot divide $x_1^{k}x_2^{s(k)}$ for any $k>1$, i.e., $s(1)>s(k)$ for any $1<k<\ell(\A)$. In other words, $\{s(1), \ldots, s(\ell(\A)-1)\}$ is a permutation of $\{1, \ldots, \ell(\A)-1\}$ and has $s(1)$ as its greatest element, hence $s(1)=\ell(\A)-1$. Therefore, $x_1x_2^{\ell(\A)-1}$ is in $M_{\A_1}^*$ and it follows from definition that $a_1\cdot a_2^{\ell(\A)-1}=1$, i.e., $a_1=a_2$. Thus $\theta=a_1=a_2$ is the desired element.

\item {\bf Case 2.} $h>2$. Again, since $\Lambda_h$ is onto, it follows that $x_ix_j^{\ell(\A)-1}$ is in $M_{\A_1}^*$ if $1\le i< j<h$. In the same way of the previous item, we conclude that $a_i=a_j$ for $1\le i<j<h$ and, in particular, $a_i=a_2$ for $1\le i<h$.
Recall that $\Lambda_h'$ is also onto and as before we can see that this implies $a_i=a_j$ if $1<i<j\le h$ and, in particular, $a_i=a_2$ for $1<i\le h$. Therefore $a_i=a_2$ for all $1\le i\le h$ and so $\theta=a_2$ is the desired element.
\end{enumerate}

Conversely, if $H(\A)=\{\theta, \ldots, \theta\}$ for some element $\theta$ of order $d=\ell(\A)$, we can easily verify that any element of $\mathcal M_{h}(\ell(\A))$ is in $M_{\A_1}^*$ and then $|\mathcal M_{h}(\ell(\A))|\le |M_{\A_1}^*|$. Since $\Lambda_h$ is one-to-one, $\Lambda_h$ must be an one-to-one correspondence (in fact, $\Lambda_h$ will be the identity map in this case). Thus $|M_{\A_1}^*|=|\mathcal M_{h}(\ell(\A))|$, i.e., $N(\A_1)=|M_{\A_1}^*|-h=\binom{\ell(\A)+h-1}{h-1}-h$.

\end{enumerate}
\end{proof}

Since the number $\ell(\A)$ defined above is always a divisor of $q-1$ and $|S_{\A}|=N(\A_1)+n$, Theorem \ref{upper-bound} implies the following:

\begin{corollary}
Let $\A\in G^n$ be an element of type $(h, t)$, where $h\ge 2$. Then $|S_{\A}|\le \binom{q+h-2}{h-1}+n-h$ with equality if and only if there exists a primitive element $\theta\in \F_{q}^*$ such that $H(\A)=\{\theta, \ldots, \theta\}$.
\end{corollary}

In the case $q=3$, notice that $-1\in \F_3$ is the only nonzero element with order greater than one. In particular, for $h\ge 2$ and $A$ an element of type $(h, t)$, $|S_{\A}|$ always attain the bound $\binom{q+h-2}{h-1}+n-h=n+\frac{h(h-1)}{2}$.

We have seen that the bounds for the number $N(\A_1)$ or even the criterion for when $R_{\A}$ is free depend only on the order of the elements in $H(\A)$. We finish this section with a simple example, showing that $N(\A_1)$ depends strongly on the elements of $H(\A)$, not only on their orders.

\begin{example}
Suppose that $q\equiv 1\pmod 8$ and let $\lambda$ be an element of order $8$ in $\F_q^*$. Let $\A$ and $\A'$ the elements of type $(2, 0)$ in $G^2$ such that $H(\A)=\{\lambda^3, \lambda^2\}$ and $H(\A')=\{\lambda^6, \lambda^7\}$. In other words, $\A\circ f(x, y)=f(\lambda^3\cdot x, \lambda^2\cdot  y)$ and $\A'\circ f(x, y)=f(\lambda^6\cdot x, \lambda^7\cdot x)$. Both sets $H(\A)$ and $H(\A')$ have an element of order $8$ and an element of order $4$. By a direct calculation we find $M_{\A_1}^*=\{x^8, x^2y, y^4\}$ and $M_{\A'_1}^*=\{x^4, xy^6, x^2y^4, x^3y^2, y^8\}$. Hence $N(\A_1)=1$ and $N(\A_1')=3$.
\end{example}

\section{Invariants through the action of Sylow subgroups of $G^n$}
In the previous sections we explore the structure of the fixed-point subring $R_{\A}$ arising from the $\F_q$-automorphism induced by an element $\A\in G^n$. In this section we consider a more restricted class of invariants. For a sugbroup $H\in G^n$, we define $R_{H}$ the set of elements in $\mathcal A_n$ that are fixed by any element $\A\in H$. In other words, $R_{H}=\{f\in \mathcal A_n\, |\, \A\circ f=f\, , \forall \A\in H\}$, is the {\it fixed-point} subring of $\mathcal A_n$ by $H$. We consider the ring of invariants $R_{H}$, for $H$ a Sylow subgroup of $G^n$.

Recall that $G^n$ has $[q(q-1)]^n$ elements and let $q-1=r_1^{\beta_1}\cdots r_s^{\beta_s}$, be the prime factorization of $q-1$, where $q$ is a power of a prime $p$ and $s=\omega(q-1)$. From definition, the Sylow subgroups of $G^n$ are those ones of order equal to the maximal power of a prime dividing $[q(q-1)]^n$; the Sylow $p$-subgroups of $G^n$ have order $q^{n}$ and, for each $1\le i\le s$, the Sylow $r_i$-subgroups of $G^n$ have order $r_i^{n\beta_i}$. It is well known that any two Sylow $r$-subgroups are conjugated and, by small modification of Theorem \ref{HT}, we see that any two conjugated groups $H, H'\in G^n$ have isomorphic fixed-point subrings. In particular, we just have to work with specific Sylow $r$-groups of $G^n$. We will naturally choose the simplest ones. 

We summarize here the ideas contained in this section. Essentially, we try to find a set of generators for $H$ such that their correspondents $\F_q$-automorphisms leave fixed all but one variable in $\{x_1, \ldots, x_n\}$. Using separation of variables (Lemma \ref{separation}) we characterize independently the rings of invariants for each automorphism. The ring $R_{H}$ will be the intersection of such rings; at this step, we follow as in the proof of Proposition \ref{fixed-ring}. For simplicity, we omit proofs that are completely analogous to the ones that we have already done.

\subsection{Homotheties and Sylow $r_i$-subgroups}
We start fixing some notation. For any nonzero element $a\in \F_q$, set $A(a)=\left(\begin{matrix}a&0\\0&1\end{matrix}\right)\in G$. For each prime $r_i$ dividing $q-1$, let $G(r_i)\le G$ be the set of matrices $A(a)$, where $a\in \F_q^*$ is such that $a^{r_i^{\beta_i}}=1$. Clearly $G(r_i)$ is a group with $r_i^{\beta_i}$ elements. Therefore, $$H(r_i):=\underbrace{G(r_i)\times \cdots \times G(r_i)}_{n\,\mathrm{times}}\le G^n$$ has order $r_i^{n\beta_i}$, i.e., $H(r_i)$ is a Sylow $r_i$-subgroup of $G^n$. If $\theta_i\in \F_q^{*}$ is an element of order $r_i^{\beta_i}$, we can verify that $H(r_i)$ is generated by $\{\A_j(\theta_i),\, 1\le j\le n\}$, where $\A_j(\theta_i)=(I, \ldots, A(\theta_i), \ldots, I)$ is the element of $G^n$ such that its $k$-th coordinate is the identity matrix $I$ for $k\ne j$ and the $j$-th coordinate of $\A_j(\theta_i)$ is the matrix $A(\theta_i)$. In particular, the $\F_q$-automorphism induced by $\A_j(\theta_i)$ fixes each variable $x_k$ for $k\ne j$ and maps $x_j$ to $\theta_ix_j$. Since $\{\A_j(\theta_i), 1\le j\le n\}$ generates $H(r_i)$, we have that $f\in R_{H(r_i)}$ if and only if
\begin{equation}\label{homo}f=\A_1(\theta_i)\circ f=\A_2(\theta_i)\circ f=\cdots =\A_n(\theta_i)\circ f.\end{equation}
In other words,
$$f(x_1, \ldots, x_n)=f(\theta_ix_1, x_2, \ldots, x_n)=f(x_1,\theta_ix_2, \ldots, x_n)=\cdots=f(x_1, x_2, \ldots, \theta_ix_n).$$

We obtain the following:

\begin{proposition}\label{homo2}
For a fixed $i$ such that $1\le i\le s=\omega(q-1)$, set $d(i)=r_i^{\beta_i}$. Then $$R_{H(r_i)}=\F_q[x_1^{d(i)}, \ldots, x_n^{d(i)}].$$ In particular, $R_{H(r_i)}$ is a free $\F_q$-algebra, isomorphic to $\mathcal A_n$.  
\end{proposition}

\begin{proof}
From Eq.~\eqref{homo} we can see that $$R_{H(r_i)}=\bigcap_{1\le j\le n}R_{\A_j(\theta_i)}.$$ 
Also, a ``translated'' version of Proposition \ref{homothety} for each $\A_j(\theta_i)$ yields $R_{\A_j(\theta_i)}=\F_q[x_1, \ldots, x_j^{d(i)}, \ldots, x_n]$. Following the proof of Proposition \ref{fixed-ring} we obtain
$$\bigcap_{1\le j\le n}R_{\A_j(\theta_i)}=\F_q[x_1^{d(i)}, \ldots, x_n^{d(i)}].$$
Therefore, $R_{H(r_i)}=\F_q[x_1^{d(i)}, \ldots, x_n^{d(i)}]$. Since $d(i)$ is a divisor of $q-1$, it follows from Corollary \ref{AI} that $R_{H(r_i)}$ is generated by $n$ algebraically independent elements of $\mathcal A_n$. In particular, $R_{H(r_i)}$ is a free $\F_q$-algebra, isomorphic to $\mathcal A_n$.  
\end{proof}

\subsection{Translations and Sylow $p$-subgroups}
For any element $a\in \F_q$, set $B(a)=\left(\begin{matrix}1&a\\0&1\end{matrix}\right)\in G$. Also, let $G(p)\le G$ be the set of matrices of the form $B(a)$ for some $a\in \F_q$. Clearly $G(p)$ is a group with $q$ elements. Hence $$H(p):=\underbrace{G(p)\times \cdots \times G(p)}_{n\,\mathrm{times}}\le G^n$$ has order $q^n$, i.e., $H(p)$ is a Sylow $p$-subgroup of $G^n$. Notice that $G(p)$ is generated by $\{\B_j(a), \, a\in \F_q, \, 1\le j\le n\}$, where $\B_j(a)=(I, \ldots, B(a), \ldots, I)$ is the element of $G^n$ such that its $i$-th coordinate is the identity matrix $I$ for $i\ne j$ and the $j$-th coordinate of $\B_j(a)$ is the matrix $B(a)$. We start looking at the case $n=1$. Notice that $B(a)\circ f=f(x+a)$. From Theorem 2.5 of \cite{LR17} we can deduce the following:

\begin{lemma}\label{translation2}
A polynomial $f(x)\in \F_q[x]$ satisfies $f(x)=f(x+b)$ for all $b\in \F_q$ if and only if $f(x)=g(x^q-x)$, for some $g(x)\in \F_q[x]$.
\end{lemma}
 
We now classify the fixed-point subring in the case of translations:

\begin{proposition}\label{trans2}
The fixed-point subring $R_{H(p)}$ of $\mathcal A_n$ by $H(p)$ satisfies $$R_{H(p)}=\F_q[x_1^q-x_1, \ldots, x_n^q-x_n].$$ In particular, $R_{H(p)}$ is a free $\F_q$-algebra, isomorphic to $\mathcal A_n$.  
\end{proposition}

\begin{proof}
The case $n=1$ follows directly from Lemma \ref{translation2}. Suppose that $n>1$. From Lemma \ref{separation}, any nonzero polynomial $f\in \mathcal A_n$ can be written uniquely as $\sum_{\alpha\in B}\x\cdot P_{\alpha}(x_1)$, where $B$ is a finite set, each $\x^{\alpha}$ is a monomial in $\F_q[x_2, \ldots, x_n]$ and $P_{\alpha}(x_1)$ is in $\F_q[x_1]$. In particular, if $f\in R_{H(p)}$, then $\B_1(a)\circ f=f$ for any $a\in \F_q$ and then
$$\sum_{\alpha\in B}\x\cdot P_{\alpha}(x_1)=\sum_{\alpha\in B}\x\cdot P_{\alpha}(x_1+a), a\in \F_q.$$
In other words, each polynomial $P_{\alpha}$ satisfies $P_{\alpha}(x)=P_{\alpha}(x+a)$ for any $a\in \F_q$. From Lemma \ref{translation2}, $P_{\alpha}$ is in $L_1:=\F_q[x_1^p-x_1]$. This shows that $R_{H(p)}\subseteq L_1[x_2, \ldots, x_n]$. In the same way, from 
$\B_i(a)\circ f=f$ for any $a\in \F_q$, we obtain $R_{H(p)}\subseteq L_i[x_1, \ldots, x_{i-1}, x_{i+1}, \ldots, x_n]$ for any $1\le i\le n$, where $L_i:=\F_q[x_i^q-x_i]$. Therefore,
$$R_{H(p)}\subseteq \bigcap_{1\le i\le n} L_i[x_1, \ldots, x_{i-1}, x_{i+1}, \ldots, x_n].$$
Following the proof of Proposition \ref{fixed-ring} we obtain
$$\bigcap_{1\le i\le n} L_i[x_1, \ldots, x_{i-1}, x_{i+1}, \ldots, x_n]=\F_q[x_1^q-x_1, \ldots, x_n^q-x_n],$$
and then $$R_{H(p)}\subseteq \F_q[x_1^q-x_1, \ldots, x_n^q-x_n].$$ The reverse inclusion is trivial. Notice that, if we set $f_i=x_i^q-x_i$, the Jacobian $\det(J(f_1, \cdots, f_n))$ equals $(-1)^n\ne 0$. It follows from the (weak) Jacobian Criterion that the polynomials $f_i$ are algebraically independent. In particular, $R_{H(p)}$ is a free $\F_q$-algebra, isomorphic to $\mathcal A_n$.  

\end{proof}

Combining Propositions \ref{homo2} and \ref{trans2} we conclude the following:

\begin{theorem}
Let $r$ be any prime dividing $p(q-1)$ and $H$ a Sylow $r$-subgroup of $G^n$. The fixed-point subring $R_{H}$ of $\mathcal A_n$ by $H$ is a free $\F_q$-algebra, isomorphic to $\mathcal A_n$.
\end{theorem}

\section{Elements of type $(2, 0)$ in $G^2$ and their fixed elements}
Let $\mathcal H$ be the subgroup of $G^2$ comprising the elements of type $(2, 0)$ in $G^2$. Throughout this section, $\A$ is always an element of $\mathcal H$, $\A=(A(a), A(b))$ for some $a, b\in \F_q^*$ and we write $\mathcal A_2=\F_q[x, y]$. We have seen that $R_{\A}$ is finitely generated and a minimal set of generators can be given by Eq.~\eqref{product-one}; the generators are of the form $x^iy^j$, where $i, j\in \N$. We have shown, with a single example, that the size of this generating set depend strongly on $a$ and $b$, not only on their orders. Here we give an alternative characterization of the invariant polynomials. 

For a monomial term $e\cdot x^iy^j$ with $e\in \F_{q}^*$, we say that $i+j\in \N$ is the degree of $e\cdot x^iy^j$. The elements of $\F_q$ are of degree zero. It turns out that $\F_q[x, y]$ is a graded ring with the following grading:
\begin{equation}\label{grading}\F_q[x, y]=\bigoplus_{n\ge 0} F_n,\end{equation}
where $F_n$ is the subset of $\F_q[x, y]$ composed by all polynomials that are sums of monomial terms of degree $n$. In particular, $F_0=\F_q$. In general, each $F_i$ is an abelian group with respect to the sum. Usually, we say that $g_n\in F_n\setminus\{0\}$ is a {\it form} of degree $n$ or an {\it homogeneous} polynomial of degree $n$. In other words, Eq.~\eqref{grading} says that any nonnzero polynomial $F\in \F_q[x, y]$ can be written uniquely as a sum $g_0+g_1+\cdots+g_n$, where each $g_i$ is either zero or a form of degree $i$ and $g_n\ne 0$. The elements of type $(2, 0)$ have a property that does not hold in the whole group $G^2$: the composition $\A\circ \x^{\alpha}$ equals $\x^{\alpha}$ times a constant. In particular, $\mathcal H$ acts on each set $F_n$ via the automorphisms $\A$. From uniqueness, we can see that $F\in \F_q[x, y]$ is fixed by $\A\in \mathcal H$ if and only if each of its homogeneous components are fixed. In particular, if $F_n(\A)$ denotes the subset of $F_n$ comprising the elements fixed by $\A$, then $R_{\A}$ is also a graded ring with the following grading:
$$R_{\A}=\bigoplus_{n\ge 0} F_n(\A).$$

Clearly $F_0(\A)=F_0=\F_q$. It is then sufficient to characterize each set $F_n(\A)$. We recall a standard property of homogeneous polynomials in two variables: for a given form $g\in \F_q[x, y]$ of degree $n$, there exists a polynomial $f\in \F_q[t]$ of degree at most $n$ such that 
\begin{equation}\label{dehomo}g=y^n\cdot f\left(\frac{x}{y}\right).\end{equation}

Also, $g$ is irreducible if and only if the univariate polynomial $f$ is irreducible and of degree $n$. Conversely, for a polynomial $f(t)\in \F_q[t]$ of degree at most $n$, we can associate to it a form $g$ of degree $n$ given by Eq.~\eqref{dehomo}. This is an one-to-one correspondence between $F_n$ and the set $\mathcal C_n$ of polynomials of degree at most $n$ in $\F_q[t]$; this correspondence induces an one-to-one correspondence between the set of irreducible polynomials in $F_n$ and the set of irreducible univariate polynomials of degree $n$ in $\F_q[t]$. As follows, we show how these correspondences behave through the action of $\mathcal H$:

\begin{lemma}\label{h:}
Let $g\in F_n\setminus \{0\}$ such that $g=y^nf\left(\frac{x}{y}\right)$, where $f\in \F_q[t]$. Also, let $\A\in \mathcal H$ and $\A=(A(a), A(b))$. Set $c=ab^{-1}$. Then $g\in F_n(\A)$ if and only if $$b^{-n}f(t)=f(ct).$$ In particular, for $n\ge 2$, $g$ is irreducible and $g\in F_n(\A)$ if and only if $f(t)\in \F_q[t]$ is irreducible of degree $n$, $b^n=1$ and $f(t)=f(ct)$.  
\end{lemma}

\begin{proof}
An easy calculation yields $\A\circ g=g(ax, by)=b^ny^nf\left(c\cdot \frac{x}{y}\right)$. In particular, $\A\circ g=g$ implies $$b^ny^nf\left(\frac{cx}{y}\right)=y^nf\left(\frac{x}{y}\right).$$
If we set $z=\frac{x}{y}$ for $x, y$ in the algebraic closure $\overline{\F}_q$ of $\F_q$, $y\ne 0$, from the previous equality we obtain $b^nf(cz)=f(z)$ for any $z\in \overline{\F}_q$. In particular we have the polynomial identity $b^{-n}f(t)=f(ct)$. The converse is trivially true. Note that $g\in F_n$ is irreducible if and only if $f(t)\in \F_q[t]$ is irreducible of degree $n$. In particular, for $n\ge 2$, $f(t)$ has nonzero constant term $a_0$ and the equality $b^{-n}f(t)=f(ct)$ yields $a_0b^n=a_0$, hence the equality $b^{-n}f(t)=f(ct)$ is equivalent to $b^n=1$ and $f(t)=f(ct)$.
\end{proof}

If $f$ is of degree $d\le n$, a comparison on the leading coefficients in the equality $b^{-n}f(t)=f(ct)$ yields $b^{-n}=c^{d}$. Conversely, if $b^{-n}=c^{d}$ for some $d\le n$, then $x^{d}y^{n-d}$ is an element of $F_n(\A)$.  Set $k=\ord(b)$, the multiplicative order of $b$. Let $S(a, b)$ be the set of nonnegative integers $n\le k$ such that there exists a nonnegative integer $d\le n$ with $b^{-n}=c^{d}$. In particular, from Lemma \ref{h:}, for each $n\in \N$, we have $F_n(\A)\ne \{0\}$ if and only if $n\equiv r\pmod k$ for some $r\in S(a, b)$ and, in this case, $F_n(\A)$ contains only elements of the form $y^nf\left(\frac{x}{y}\right)$, where either $f=0$ or $\deg(f)\le n$. If $[n]$ denotes the least nonnegative integer $r$ such that $n\equiv r\pmod k$, we may rewrite 
$$R_{\A}=\bigoplus_{[n]\in S(a, b)\atop n\in \N} F_n(\A),$$
where $F_n(\A)$ comprises the polynomials of the form $y^nf\left(\frac{x}{y}\right)$ for some $f(t)\in \F_q[t]$ of degree $d\le n$, such that either $f=0$ or $c^d=b^{-n}$ and $f(ct)=b^{-n}f(t)$.

Let $n$ be a positive integer such that $[n]\in S(a, b)$ and let $d_0\le n$ be the least nonnegative integer such that $c^{d_0}=b^{-n}$ . For a nonnegative integer $d$ such that $d\le n$, we have $c^d=b^{-n}$ if and only if $d\equiv d_0\pmod \ell$, where $\ell=\ord(c)$. In particular, $d=d_0+s\cdot \ell$ for some $0\le s\le s(n)$, where $s(n)=\lfloor\frac{n-d_0}{\ell}\rfloor$.

If $f(t)=a_dt^d+\cdots+a_it+a_0$ is a polynomial of degree $d=d_0+s\cdot \ell$, the equality $f(ct)=b^{-n}f(t)$ yields $a_ib^{-n}=a_ic^i$, i.e., $$a_i=c^{d_0-i}a_i.$$ The last equality is equivalent to $a_i=0$ for $c^{d_0-i}\ne 1$, i.e., $i\not\equiv d_0\pmod \ell$. This shows that $f(ct)=b^{-n}f(t)$ if and only if there exists a polynomial $h(t)\in \F_q[t]$ such that $f(t)=t^{d_0}h(t^{\ell})$. The degree of $h(t)$ equals $\frac{d-d_0}{\ell}=s$ and, in particular, there are exactly $(q-1)q^{s}$ choices for $h(t)$. 

Combining all those observations, we conclude the following:
\begin{theorem}
Let $n$ be a positive integer such that $[n]\in S(a, b)$ and $b^{-n}=c^{d_0}$. The set $F_n(\A)$ comprises the polynomials of the form $y^{n-d_0}x^{d_0}h\left(\frac{x^{\ell}}{y^{\ell}}\right)$, where $h=0$ or $h$ is an univariate polynomial of degree at most $s(n)$. In particular, the size of $F_n(\A)$ equals
$$1+(q-1)\sum_{0\le s\le s(n)}q^{s}=q^{s(n)+1}.$$
\end{theorem}
Let us see what happens with the irreducible homogeneous: recall that, for $n\ge 2$, $g=y^nf\left(\frac{x}{y}\right)$ is irreducible and fixed by $\A$ if and only if $f$ is irreducible of degree $n$, $b^n=1$ and $f(t)=f(ct)$. If we write $f(x)=a_nx^n+\cdots+a_1x+a_0$, equality $f(t)=f(ct)$ yields $a_i=c^ia_i$. The last equality is equivalent to $a_i=0$ for $i\not\equiv 0\pmod \ell$, i.e., $f(t)=h(t^{\ell})$ for some polynomial $h\in \F_q[t]$ and then $n$ must be divisible by $\ell$. In particular, $c^n=1$ and, since $b^n=1$, it follows that $a^n=1$ and then $n$ must be divisible by $\mathrm{lcm}(\ord(a), \ord(b))=\ord(\A)$. Writing $D=\ord(\A)$ and $n=Dm$, we see that  $g=y^nf\left(\frac{x}{y}\right)$ is irreducible and fixed by $\A$ if and only if $f(t)=h(t^{\ell})$ is irreducible, where $h(t)$ has degree $\frac{Dm}{\ell}$. 

According to Theorem 3 of \cite{cohen}, the number of {\bf monic} irreducible polynomials of the form $h(t^{\ell})$, where $h(t)$ has degree $\frac{Dm}{\ell}$, equals
$$\frac{\Phi(\ell)}{Dm}\sum_{d|\frac{Dm}{\ell}\atop \gcd(d, \ell)=1}\mu(d)(q^{\frac{Dm}{d\ell}}-1),$$
where $\Phi$ is the {\it Euler Phi} function and $\mu$ is the {\it Mobius} function. For $n\ge 2$, let $N(\A, n)$ be the number of irreducible homogeneous polynomials of degree $n$ that are invariant by $\A\in \mathcal H$. In particular, $N(\A, n)=0$ if $n$ is not divisible by $D=\ord(\A)$ and, for $n=Dm$,
$$N(\A, Dm)=(q-1)\cdot \frac{\Phi(\ell)}{Dm}\sum_{d|\frac{Dm}{\ell}\atop \gcd(d, \ell)=1}\mu(d)(q^{\frac{Dm}{d\ell}}-1).$$

After simple calculations we obtain $N(\A, Dm)=\frac{(q-1)\Phi(\ell)}{Dm}q^{\frac{Dm}{\ell}}+R(q, m)$, where $$|R(q, m)|\le q^{\frac{Dm}{2\ell}+1}.$$

Also, for $n=Dm$, it follows from definition that $s(n)=\frac{n}{\ell}=\frac{Dm}{\ell}$, hence $|F_n(\A)|=q^{\frac{Dm}{\ell}+1}$ and
$$\frac{N(\A, Dm)}{|F_n(\A)|}=\frac{q-1}{q}\cdot \frac{\Phi(\ell)}{Dm}+d(q, m),$$
where $|d(q, m)|\le q^{-\frac{Dm}{2\ell}}$. In particular, for $q$ large, the density of the irreducible polynomials in $F_{Dm}(\A)$ is close to $\frac{\Phi(\ell)}{Dm}$.

Recall that, for $B=
\left(\begin{matrix}
a_1&a_2\\
a_3&a_4
\end{matrix}\right)
$ in $\GL_2(\F_q)$ and a polynomial $f(t)\in \F_q[t]$ of degree $n$, $B\diamond f=(a_3t+a_4)^nf\left(\frac{a_1t+a_2}{a_3t+a_4}\right)$. In particular, for $A(c):=\left(\begin{matrix}
c&0\\
0&1
\end{matrix}\right)$, we have $A(c)\diamond f=f(ct)$ and then, the irreducible elements of $F_{Dm}(\A)$ are of the form $y^{Dm}f\left(\frac{x}{y}\right)$, where $f$ is an irreducible polynomial of degree $Dm$ in $\F_q[t]$ such that $A(c)\diamond f=f$. 

According to Theorem 3 of \cite{Gar11}, the irreducible polynomials $f$ of degree $Dm$ such that $A(c)\diamond f=f$ are exactly the irreducible factors of degree $Dm$ of the polynomials $ct^{q^r-1}-1, r>0$. In particular, the irreducible elements of $F_{Dm}(\A)$ are exactly the irreducible (homogeneous) factors of degree $Dm$ of
$$ax^{q^r-1}-by^{q^r-1}, r>0.$$ 

\subsection{The compositions $A\diamond f$ and invariant homogeneous of $\F_q[x, y]$}
There is a more general action of $\GL_2(\F_q)$ on the polynomial ring $\F_q[x, y]$: given $A=\left(\begin{matrix}
a_1&a_2\\
a_3&a_4
\end{matrix}\right)\in\GL_2(\F_q)$ and $f\in \F_q[x, y]$, we define $$A\ast f=f(a_1x+a_2y, a_3x+a_4y).$$ 

Notice that, considering the diagonal elements of $\GL_2(\F_q)$, this is the action of $\mathcal H$ previously considered: for $A=\left(\begin{matrix}
a&0\\
0&b
\end{matrix}\right) \in\GL_2(\F_q)$ and $\A=(A(a), A(b))\in \mathcal H$, we have $A\ast f=\A\circ f$.

This is a classical group action in the {\it Invariant Theory of Finite Groups}. It can be verified that, for any $A\in \GL_2(\F_q)$, the composition $A\ast f$ preserves the degree of the homogeneous components of $f$: in fact, $\GL_2(\F_q)$ acts on the set $F_n$ of homogeneous polynomials of degree $n$ in $\F_q[x, y]$ via the compositions $A\ast f$. For more details, see Chapter 7 of \cite{Cox}.

In particular, if $\mathcal R_A$ denotes the ring of invariants and $\mathcal F_n(A)$ denotes the set of homogeneous invariants of degree $n$, $\mathcal R_A$ is a graded ring with the following grading:
$$\mathcal R_{A}=\bigoplus_{n\ge 0} \mathcal F_n(A).$$

What about the components $\mathcal F_n(A)$? Recall that any element $f\in F_n$ can be written uniquely as $y^ng(\frac{x}{y})$, for some $g(t)\in \F_q[t]$ of degree at most $n$. In particular, $A\ast f=f$ if and only if

$$(a_3x+a_4y)^ng\left(\frac{a_1x+a_2y}{a_3x+a_4y}\right)=y^ng\left(\frac{x}{y}\right).$$
In a similar way as before (setting $t=x/y$), we see that the previous equality is equivalent to the polynomial identity \begin{equation}\label{multiaction}(a_3t+a_4)^ng\left(\frac{a_1t+a_2}{a_3t+a_4}\right)=g(t).\end{equation}
If $a_3=a_2=0$, we are back to the diagonal case. For $a_3\ne 0$, $g$ must be of degree $n$ and Eq.~\eqref{multiaction} is equivalent to $A\diamond g=g$. In this case, we have an one-to-one correspondence between the elements of $\mathcal F_n(A)$ and the univariate polynomials $g$ of degree $n$ such that $A\diamond g=g$; clearly this induces a correspondence on the respective irreducible polynomials. 

For the case $a_2\ne 0$, note that any homogeneous $f\in F_n$ can be written uniquely as $x^nh\left(\frac{y}{x}\right)$, where $h(t)\in \F_q[t]$ is a polynomial of degree at most $n$:  in this case, we obtain a condition like Eq.~\eqref{multiaction}. In fact, we have an one-to-one correspondence between the elements of $\mathcal F_n(A)$ and the univariate polynomials $h$ of degree $n$ such that $\overline{A}\diamond h=h$, where $\overline{A}\in \GL_2(\F_q)$ is the matrix obtained after interchanging the arrows of $A$.

For instance, consider $\mathcal B=\left(\begin{matrix}
0&1\\
1&0
\end{matrix}\right)
$. Note that $\mathcal B\ast f=f(y, x)$; this is the only nontrivial permutation of the variables. The fixed-point subring comprises the {\it symmetric} polynomials in two variables. It is well known that, in this case, the fixed-point subring is generated by the symmetric polynomials $\sigma_1=x+y$ and $\sigma_2=xy$, i.e., $\mathcal R_{\mathcal B}=\F_q[x+y, xy]$. Here we may characterize the fixed-point subring by their grading components: since $\mathcal B\diamond g$ is the {\it reciprocal} of $g$, $\mathcal F_n(\mathcal B)$ comprises the polynomials of the form $y^ng(x/y)$, where $g(t)\in \F_q[t]$ is a {\it self-reciprocal} polynomial of degree $n$. 

In general, the study of Eq.~\eqref{multiaction} yields a characterization of the elements contained in each component $\mathcal F_n(A)$; from this characterization, we may find enumeration formulas for the number of (irreducible) polynomials in $\mathcal F_n(A)$.
\section{Conclusions}
We have noticed that, for  
$$G=\left\{\left(\begin{matrix}
a&b\\
0&1
\end{matrix}\right), a, b\in \F_q, a\ne 0 \right\},
$$ 
the group $G^n\subset \GL_2(\F_q)^n$ acts on the ring of polynomials in $n$ variables over $\F_q$. For $\A\in G^n$, we have explored the algebraic properties of the fixed-point subring $R_{\A}$. In particular, we have seen that $R_{\A}$ is always finitely generated as an $\F_q-$algebra and a minimal generating set $S_{\A}$ for $R_{\A}$ can be explicitly computed. We have given a criteria for when $R_{\A}$ is free, we have provided upper bounds for the size of $S_{\A}$ and characterized the elements $\A$ for which this bound is attained. In our approach, some algebraic structures of $R_{\A}$ are naturally related to other topics, such as the action studied in \cite{Gar11} and minimal product-one sequences in abelian groups.

\begin{center}{\bf Acknowledgments}\end{center}
This work was conducted during a scholarship supported by the Program CAPES-PDSE (process - 88881.134747/2016-01) at Carleton University. Financed by CAPES - Brazilian Federal
Agency for Support and Evaluation of Graduate Education within the Ministry of Education of Brazil.


\end{document}